\documentclass[11pt]{article}
\usepackage{amssymb}
\usepackage{amsmath}
\usepackage{amsfonts}
\usepackage{natbib}
\usepackage{hyperref}
\usepackage{bbm,color,enumerate,mathrsfs,natbib,graphicx}
\newtheorem{theorem}{Theorem}[section]

\newtheorem{condition}{Condition}[section]
\newtheorem{lemma}{Lemma}[section]

\newtheorem{assumption}{Assumption}

\newtheorem{remark}{Remark}[section]
\newtheorem{corollary}{Corollary}[section]

\newenvironment{proof}[1][Proof]{\noindent \textbf{#1.} }{\  \rule{0.5em}{0.5em}}

\begin{document}

\title{Optimal Uniform Convergence Rates for \\Sieve Nonparametric Instrumental Variables Regression\thanks{%
Support from the Cowles Foundation is gratefully acknowledged. We thank conference participants of SETA2013 in Seoul and AMES2013 in Singapore for useful comments. Any errors are the
responsibility of the authors.}}
\author{Xiaohong Chen\thanks{%
Cowles Foundation for Research in Economics, Yale
University: \texttt{xiaohong.chen@yale.edu}} \ and Timothy M. Christensen%
\thanks{%
Department of Economics, Yale University: \texttt{timothy.christensen@yale.edu}}}
\date{First version January 2012; Revised August 2013}
\maketitle

\begin{abstract}
\noindent We study the problem of nonparametric regression when the regressor is endogenous, which is an important nonparametric instrumental variables (NPIV) regression in econometrics and a difficult ill-posed inverse problem with unknown operator in statistics. We first establish a general upper bound on the sup-norm (uniform) convergence rate of a sieve estimator, allowing for endogenous regressors and weakly dependent data. This result leads to the optimal sup-norm convergence rates for spline and wavelet least squares regression estimators under weakly dependent data and heavy-tailed error terms. This upper bound also yields the sup-norm convergence rates for sieve NPIV estimators under i.i.d. data: the rates coincide with the known optimal $L^2$-norm rates for severely ill-posed problems, and are power of $\log(n)$ slower than the optimal $L^2$-norm rates for mildly ill-posed problems. We then establish the minimax risk lower bound in sup-norm loss, which coincides with our upper bounds on sup-norm rates for the spline and wavelet sieve NPIV estimators. This sup-norm rate optimality provides another justification for the wide application of sieve NPIV estimators. Useful results on weakly-dependent random matrices are also provided.


\end{abstract}
\thispagestyle{empty}
\setcounter{page}{0}

\newpage

\section{Introduction}

In economics and other social sciences one frequently encounters the relation
\begin{equation}\label{basic}
 Y_{1i} = h_0(Y_{2i}) + \epsilon_i
\end{equation}
where $Y_{1i}$ is a response variable, $Y_{2i}$ is a predictor variable, $h_0$ is an unknown structural function of interest, and $\epsilon_i$ is an error term. However, a latent external mechanism may ``determine'' or ``cause'' $Y_{1i}$ and $Y_{2i}$ simultaneously, in which case the conditional mean restriction $E[\epsilon_i |Y_{2i}] = 0$ fails and $Y_{2i}$ is said to be \emph{endogenous}.\footnote{In a canonical example of this relation, $Y_{1i}$ may be the hourly wage of person $i$ and $Y_{2i}$ may include the education level of person $i$. The latent ability of person $i$ affects both $Y_{1i}$ and $Y_{2i}$. See \cite{BlundellPowell} for other examples and discussions of endogeneity in semi/nonparametric regression models.} When the regressor $Y_{2i}$ is endogenous one cannot use standard nonparametric regression techniques to consistently estimate $h_0$. In this instance one typically assumes that there exists a vector of \emph{instrumental variables} $X_i$ such that $E[\epsilon_i |X_i] = 0$ and for which there is a nondegenerate relationship between $X_i$ and $Y_{2i}$. Such a setting permits estimation of $h_0$ using nonparametric instrumental variables (NPIV) techniques based on a sample $\{(X_i,Y_{1i},Y_{2i})\}_{i=1}^n$. In this paper we assume that the data is strictly stationary in that $(X_i,Y_{1i},Y_{2i})$ has the same (unknown) distribution $F_{X,Y_1,Y_2}$ as that of $(X,Y_{1},Y_{2})$ for all $i$.\footnote{The subscript $i$ denotes either the individual $i$ in a cross-sectional sample or the time period $i$ in a time-series sample. Since the sample is strictly stationary we sometimes drop the subscript $i$ without confusion.}

NPIV estimation has been the subject of much research in recent years, both because of its practical importance to applied economics and its prominent role in the literature on linear ill-posed inverse problems with unknown operators. In many economic applications the joint distribution $F_{X,Y_2}$ of $X_i$ and $Y_{2i}$ is unknown but is assumed to have a continuous density. Therefore the conditional expectation operator $Th(\cdot)=E[h(Y_{2i})|X_i = \cdot ]$ is typically unknown but compact. Model (\ref{basic}) with $E[\epsilon_i |X_i] = 0$ can be equivalently written as
\begin{equation} \label{knownT} \begin{array}{rcl}
 Y_{1i} & = & Th_0(X_i) + u_i \\
 E[u_{i}|X_i] &= & 0 \end{array}
\end{equation}
where $u_i = h_0(Y_{2i}) - Th_0(X_i) + \epsilon_i$. Model (\ref{knownT}) is called the reduced-form NPIV model if $T$ is assumed to be unknown and the nonparametric indirect regression (NPIR) model if $T$ is assumed to be known. Let $\widehat{E}[Y_{1}|X= \cdot]$ be a consistent estimator of $E[Y_{1}|X= \cdot]$. Regardless of whether the compact operator $T$ is unknown or known, nonparametric recovery of $h_0$ by inversion of the conditional expectation operator $T$ on the left-hand side of the Fredholm equation of the first kind
\begin{equation} \label{fredholm}
 Th( \cdot) = \widehat{E}[Y_{1}|X= \cdot]
\end{equation}
leads to an ill-posed inverse problem (see, e.g., \cite{Kress}). Consequently, some form of regularization is required for consistent nonparametric estimation of $h_0$. In the literature there are several popular methods of NPIV estimation, including but not limited to (1) finite-dimensional sieve minimum distance estimators \citep*{NeweyPowell,AiChen2003,Blundell2007}; (2) kernel-based Tikhonov regularization estimators \citep*{HallHorowitz,Darollesetal2011,GagliardiniScaillet} and their Bayesian version \citep{FlorensSimoni}; (3) orthogonal series Tikhonov regularization estimators \citep{HallHorowitz}; (4) orthogonal series Galerkin-type estimators \citep{Horowitz2011}; (5) general penalized sieve minimum distance estimators \citep{ChenPouzo2012} and their Bayesian version \citep{LiaoJiang}. See \cite{Horowitz2011} for a recent review and additional references.

To the best of our knowledge, all the existing works on convergence rates for various NPIV estimators have only studied $L^2$-norm convergence rates. In particular, \cite{HallHorowitz} are the first to establish the minimax risk lower bound in $L^2$-norm loss for a class of mildly ill-posed NPIV models, and show that their estimators attain the lower bound. \cite{ChenReiss} derive the minimax risk lower bound in $L^2$-norm loss for a large class of NPIV models that could be mildly or severely ill-posed, and show that the sieve minimum distance estimator of \cite*{Blundell2007} achieves the lower bound. Subsequently, some other NPIV estimators listed above have also been shown to achieve the optimal $L^2$-norm convergence rates. As yet there are no published results on sup-norm (uniform) convergence rates for any NPIV estimators, nor results on what are the minimax risk lower bounds in sup-norm loss for any class of NPIV models.

Sup-norm convergence rates for any estimators of $h_0$ are important for constructing uniform confidence bands for the unknown $h_0$ in  NPIV models and for conducting inference on nonlinear functionals of $h_0$, but are currently missing. In this paper we study the uniform convergence properties of the sieve minimum distance estimator of $h_0$ for the NPIV model, which is a nonparametric series two-stage least squares regression estimator \citep*{NeweyPowell,AiChen2003,Blundell2007}. We focus on this estimator because it is easy to compute and has been used in empirical work in demand analysis \citep*{Blundell2007,ChenPouzo2009}, asset pricing \citep{Chen2009}, and other applied fields in economics. Also, this class of estimators is known to achieve the optimal $L^2$-norm convergence rates for both mildly and severely ill-posed NPIV models.

We first establish a general upper bound (Theorem \ref{sup norm rate gen new}) on the uniform convergence rate of a sieve estimator, allowing for endogenous regressors and weakly dependent data.  To provide sharp bounds on the sieve approximation error or ``bias term'' we extend the proof strategy of \cite{Huang2003} for sieve nonparametric least squares (LS) regression to the sieve NPIV estimator. Together, these tools yield sup-norm convergence rates for the spline and wavelet sieve NPIV estimators under i.i.d. data. Under conditions similar to those for the $L^2$-norm convergence rates for the sieve NPIV estimators, our sup-norm convergence rates coincide with the known optimal $L^2$-norm rates for severely ill-posed problems, and are power of $\log( n)$ slower than the optimal $L^2$-norm rates for mildly ill-posed problems.  We then establish the minimax risk lower bound in sup-norm loss for $h_0$ in a NPIR model (i.e., (\ref{knownT}) with a known compact $T$) uniformly over H\"older balls, which in turn provides a lower bound in sup-norm loss for $h_0$ in a NPIV model uniformly over H\"older balls. The lower bound is shown to coincide with our sup-norm convergence rates for the spline and wavelet sieve NPIV estimators.

To establish the general upper bound, we first derive a new exponential inequality for sums of weakly dependent random matrices in Section \ref{ei sec}. This allows us to weaken conditions under which the optimal uniform convergence rates can be obtained.
As an indication of the sharpness of our general upper bound result, we show that it leads to the optimal uniform convergence rates for spline and wavelet LS regression estimators with weakly dependent data and heavy-tailed error terms. Precisely, for beta-mixing dependent data and finite $(2+\delta )$-th moment error term (for $\delta \in (0,2)$), we show that the spline and wavelet nonparametric LS regression estimators attain the minimax risk lower bound in sup-norm loss of \cite{Stone1982}. This result should be very useful to the literature on nonparametric estimation with financial time series.

The NPIV model falls within the class of statistical linear ill-posed inverse problems with \emph{unknown} operators and additive noise. There is a vast literature on statistical linear ill-posed inverse problems with known operators and additive noise. Some recent references include but are not limited to
\citet*{Cavalieretal2002}, \citet*{Cohenetal2004} and \cite{Cavalier2008}, of which density deconvolution is an important and extensively-studied problem (see, e.g., \cite{CarrollHall,Zhang,Fan1991,HallMeister,LouniciNickl}). There are also papers on statistical linear ill-posed inverse problems with pseudo-unknown operators (i.e., known eigenfunctions but unknown singular values) (see, e.g., \cite{CavalierHengartner}, \cite{LoubesMarteau}). Related papers that allow for an unknown linear operator but assume the existence of an estimator of the operator (with rate) include \cite{EfromovichKoltchinskii}, \cite{HoffmannReiss} and others.
To the best of our knowledge, most of the published works in the statistical literature on linear ill-posed inverse problems also focus on the rate optimality in $L^2$-norm loss, except that of \cite{LouniciNickl} which recently establishes the optimal sup-norm convergence rate for a wavelet density deconvolution estimator. Therefore, our minimax risk lower bounds in sup-norm loss for the NPIR and NPIV models also contribute to the large literature on statistical ill-posed inverse problems.

The rest of the paper is organized as follows. Section \ref{npiv sec} outlines the model and presents a general upper bound on the uniform convergence rates for a sieve estimator. Section \ref{o-npiv sec} establishes the optimal uniform convergence rates for the sieve NPIV estimators, allowing for both mildly and severely ill-posed inverse problems. Section \ref{reg sec} derives the optimal uniform convergence rates for the sieve nonparametric (least squares) regression, allowing for dependent data. Section \ref{ei sec} provides useful exponential inequalities for sums of random matrices, and the reinterpretation of equivalence of the theoretical and empirical $L^2$ norms as a criterion regarding convergence of a random matrix. The appendix contains a brief review of the spline and wavelet sieve spaces, proofs of all the results in the main text, and supplementary results.

\paragraph{Notation:} $%
\|\cdot\|$ denotes the Euclidean norm when applied to vectors and the matrix
spectral norm (largest singular value) when applied to matrices. For a random variable $Z$ let $%
L^q(Z) $ denote the spaces of (equivalence classes of) measurable functions
of $z$ with finite $q$-th moment if $1 \leq q < \infty$ and let $\|\cdot\|_{L^q (Z)}$ denote the $L^q (Z)$ norm. Let $L^\infty(Z)$ denote the space of measurable functions of $z$ with finite sup norm $\|\cdot\|_\infty$. If $A$ is a square matrix, $\lambda_{\min}( A)$
and $\lambda_{\max}(A)$  denote its smallest and largest eigenvalues,
respectively, and $A^-$ denotes its Moore-Penrose generalized inverse.  If $\{a_n:n \geq 1\}$ and $\{b_n : n \geq 1\}$ are two sequences of non-negative numbers, $a_n \lesssim b_n$ means there exists a finite positive $C$ such that $a_n \leq C b_n$ for all $n$ sufficiently large, and $a_n \asymp b_n$ means $a_n \lesssim b_n$ and $b_n \lesssim a_n$. $\#(\mathcal S)$ denotes the cardinality of a set $\mathcal S$ of finitely many elements. Let $\mbox{BSpl}(K,[0,1]^d, \gamma )$ and $\mbox{Wav}(K,[0,1]^d, \gamma )$ denote tensor-product B-spline (with smoothness $\gamma$) and wavelet (with regularity $\gamma$) sieve spaces of dimension $K$ on $[0,1]^d$ (see Appendix \ref{sieve def} for details on construction of these spaces).

\section{Uniform convergence rates for sieve NPIV estimators}\label{npiv sec}

We begin by considering the NPIV model
\begin{equation}  \label{npiv}
\begin{array}{rcl}
Y_{1i} & = & h_0(Y_{2i}) + \epsilon_i \\
E[\epsilon_i |X_i] & = & 0%
\end{array}%
\end{equation}
where $Y_1 \in  \mathbb R$ is a response variable, $Y_2$ is an endogenous regressor with support $\mathcal Y_2 \subset \mathbb R^d$ and $X$ is a vector of conditioning variables (also called instruments) with support $\mathcal{X }\subset
\mathbb{R}^{d_x}$. The object of interest is the unknown
structural function $h_0 : \mathcal{Y}_2 \to \mathbb{R}$ which belongs to some infinite-dimensional parameter space $\mathcal H \subset L^2(Y_2)$. It is assumed hereafter that $h_0$ is identified uniquely by the conditional moment restriction (\ref{npiv}). See \cite%
{NeweyPowell}, \citet*{Blundell2007}, \citet*{Darollesetal2011}, \cite{Andrews2011}, \cite{D'Haultfoeuille}, \citet*{Chen2012b} and references therein for sufficient conditions for identification.

\subsection{Sieve NPIV estimators}

The sieve NPIV estimator due to \cite{NeweyPowell}, \cite{AiChen2003}, and \cite*{Blundell2007} is a nonparametric series two-stage least squares estimator. Let the sieve spaces $\{\Psi _{J}:J\geq 1\} \subseteq L^2(Y_2)$ and $\{B_K : K \geq 1\} \subset L^2(X)$ be sequences of subspaces of dimension $J$ and $K$ spanned by sieve basis functions such that $\Psi_J$ and $B_K$ become dense in $\mathcal H \subset L^{2}(Y_{2})$ and $L^2(X)$ as $J,K\rightarrow \infty $. For given $J$ and $K$, let $\{\psi_{J1},\ldots,\psi_{JJ}\}$ and $\{b_{K1},\ldots,b_{KK}\}$ be sets of sieve basis functions whose closed linear span generates $\Psi_J$ and $B_K$ respectively. We consider sieve spaces generated by spline, wavelet or other Riesz basis functions that have nice approximation properties (see Section \ref{o-npiv sec} for details).

In the first stage, the conditional moment function $%
m(x,h):\mathcal{X}\times \mathcal{H}\rightarrow \mathbb{R}$ given by
\begin{equation}
m(x,h)=E[Y_{1}-h(Y_{2})|X=x]
\end{equation}%
is estimated using the series (least squares) regression estimator
\begin{equation} \label{mhat equation}
\widehat{m}(x,h)=\sum_{i=1}^{n}b^{K}(x)^{\prime }(B^{\prime
}B)^{-}b^{K}(X_{i})(Y_{1i}-h(Y_{2i}))
\end{equation}%
where
\begin{equation} \begin{array}{rcl}
b^K(x) & = & (b_{K1}(x),\ldots,b_{KK}(x))' \\
B & = & (b^K(X_1),\ldots,b^K(X_n))' \,.
\end{array} \end{equation}
The sieve NPIV estimator $\widehat h$ is then defined as the solution to the second-stage minimization problem
\begin{equation}
\widehat{h}=\arg \min_{h\in \Psi_{J}}\frac{1}{n}\sum_{i=1}^{n}%
\widehat{m}(X_{i},h)^{2}
\end{equation}%
which may be solved in closed form to give
\begin{equation}
 \widehat h(y_2) = \psi^J(y_2)'[\Psi'B(B'B)^-B'\Psi]^- \Psi'B (B'B)^- B'Y
\end{equation}
where
\begin{equation}
\begin{array}{rcl}
 \psi^J(y_2) & = & (\psi_{J1}(y_2),\ldots,\psi_{JJ}(y_2))' \\
  \Psi & = & (\psi^J(Y_{21}),\ldots,\psi^J(Y_{2n}))' \\
 Y & = & (Y_{11},\ldots,Y_{1n})'\,.
\end{array}
\end{equation}%
Under mild regularity conditions (see \cite{NeweyPowell}, \citet*{Blundell2007} and \cite{ChenPouzo2012}), $\widehat h$ is a consistent estimator of $h_0$ (in both $\|\cdot\|_{L^2(Y_2)}$ and $\|\cdot\|_\infty$ norms) as $n,J,K \to \infty$, provided $J \leq K$ and $J$ increases appropriately slowly so as to  regularize the ill-posed inverse problem.\footnote{Here we have used $K$ to denote the ``smoothing parameter'' (i.e. the dimension of the sieve space used to estimate the conditional moments in (\ref{mhat equation})) and $J$ to denote the ``regularization parameter'' (i.e. the dimension of the sieve space used to approximate the unknown $h_0$). Note that \cite{ChenReiss} use $J$ and $m$, \citet*{Blundell2007} and \cite{ChenPouzo2012} use $J$ and $k$ to denote the smoothing and regularization parameters, respectively.} We note that the modified sieve estimator (or orthogonal series Galerkin-type estimator) of \cite{Horowitz2011} corresponds to the sieve NPIV estimator with $J=K$ and $\psi^J (\cdot) = b^K (\cdot)$ being orthonormal basis in $L^2 (Lebesgue)$.

\subsection{A general upper bound on uniform convergence rates for sieve estimators}

We first present a general calculation for sup-norm convergence which will be used to obtain uniform convergence rates for both the sieve NPIV and the sieve LS estimators below.

As the sieve estimators are invariant to an invertible transformation of the sieve basis functions, we re-normalize the sieve spaces $B_K$ and $\Psi_J$ so that $\{\widetilde b_{K1},\ldots,\widetilde b_{KK}\}$ and $\{\widetilde \psi_{J1},\ldots,\widetilde \psi_{JJ}\}$ form orthonormal bases for $B_K$ and $\Psi_J$. This is achieved by setting $\widetilde b^K(x) = E[b^K(X)b^K(X)']^{-1/2}b^K(x)$ where $^{-1/2}$ denotes the inverse of the positive-definite matrix square root (which exists under Assumption \ref{b sieve}(ii) below), with $\widetilde \psi^J$ similarly defined. Let
\begin{equation} \begin{array}{rcl}
 \widetilde B & = & (\widetilde b^K(X_1),\ldots,\widetilde b^K(X_n))' \\
 \widetilde \Psi & = & (\widetilde \psi^J(Y_{21}),\ldots,\widetilde \psi^J(Y_{2n}))'
\end{array}
\end{equation}
and define the $J \times K$ matrices
\begin{equation} \begin{array}{rcl} \label{S def}
 S & = & E[\widetilde \psi^J(Y_2)\widetilde b^K(X)'] \\
 \widehat S & = & \widetilde \Psi'\widetilde B/n\,. \end{array}
\end{equation}
Let $\sigma_{JK}^2 = \lambda_{\min}(SS')$. For each $h \in \Psi_J$ define
\begin{equation} \label{delta tilde m}
 \Pi_K Th(\cdot) = \widetilde b^K(x)'E[\widetilde b^K(X) (T h)(X)] = \widetilde b^K(\cdot)'E[\widetilde b^K(X) h(Y_2)]
\end{equation}
which is the $L^2(X)$ orthogonal projection of $T h(\cdot)$ onto $B_K$. The variational characterization of singular values gives
\begin{equation}
 \sigma_{JK} = \inf_{h \in \Psi_J : \|h\|_{L^2(Y_2)} = 1} \|\Pi_K Th\|_{L^2(X)} \leq 1\,.
\end{equation}
Finally, define $P_n$ as the second-stage empirical projection operator onto the sieve space $\Psi_J$ after projecting onto the instrument space $B_K$, viz.
\begin{equation}
 P_n h_0(y_2) = \widetilde \psi^J(y_2) [\widehat S(\widetilde B'\widetilde B/n)^{-} \widehat S']^{-} \widehat S (\widetilde B'\widetilde B/n)^{-} \widetilde B' H_0/n
\label{Pn def}
\end{equation}
where $H_0 = (h_0(Y_{21}),\ldots,h_0(Y_{2n}))'$.

We first decompose the sup-norm error as
\begin{equation}
 \|h_0 - \widehat h\|_{\infty} \leq \|h_0 - P_n h_0  \|_{\infty} + \|P_n h_0  - \widehat h \|_{\infty}
\end{equation}
and calculate the uniform convergence rate for the ``variance term'' $\|\widehat h - P_n h_0\|_{\infty}$ in this section. Control of the ``bias term'' $\|h_0 - P_n h_0  \|_{\infty}$ is left to the subsequent sections, which will be dealt with under additional regularity conditions for the NPIV model and the LS regression model separately.

Let $Z_i = (X_i,Y_{1i},Y_{2i})$ and $\mathcal F_{i-1} = \sigma(X_i,X_{i-1},\epsilon_{i-1},X_{i-2},\epsilon_{i-2},\ldots)$.

\begin{assumption}\label{data}
(i) $\{Z_i\}_{i=-\infty}^\infty$ is strictly stationary, (ii) $X$ has support $\mathcal X = [0,1]^{d}$ and $Y_2$ has support $\mathcal Y_2 = [0,1]^d$, (iii) the distributions of $X$ and $Y_2$ have density (with respect to Lebesgue measure) which is uniformly bounded away from zero and infinity over $\mathcal X$ and $\mathcal Y_2$ respectively.
\end{assumption}

The results stated in this section do not actually require that $\dim(X) = \dim(Y_2)$. However, most published papers on NPIV models assume $\dim(X) = \dim(Y_2) = d$ and so we follow this convention in Assumption 1(ii).

\begin{assumption}\label{resid}
(i) $(\epsilon_i,\mathcal F_{i-1})_{i=-\infty}^\infty$ is a strictly stationary martingale difference sequence, (ii) the conditional second moment $E[\epsilon_i^2 |\mathcal F_{i-1}]$ is uniformly bounded away from zero and infinity, (iii) $E[|\epsilon_i|^{2+\delta}] < \infty$ for some $\delta > 0$.
\end{assumption}

\begin{assumption}\label{f sieve}
(i) Sieve basis $\psi^J(\cdot) $ is H\"older continuous with smoothness $\gamma > p$ and $\sup_{y_2 \in \mathcal Y_2} \|\psi^J(y_2)\| \lesssim \sqrt J$, (ii)
$ \lambda_{\min}(E[\psi^J(Y_2)\psi^J(Y_2)^{\prime }]) \geq \underline
\lambda> 0
$ for all $J \geq 1$.
\end{assumption}

In what follows, $p>0$ indicates the smoothness of the function $h_0 (\cdot)$ (see Assumption \ref{parameter regression} in Section \ref{o-npiv sec}).

\begin{assumption}\label{b sieve}
(i) Sieve basis $b^K(\cdot) $ is H\"older continuous with smoothness ${\gamma}_x \geq \gamma > p$ and $\sup_{x \in \mathcal X} \| b^K(x)\| \lesssim  \sqrt K$, (ii)
$
 \lambda_{\min}(E[b^K(X)b^K(X)^{\prime }]) \geq \underline
\lambda> 0
$
for all $K \geq 1$.
\end{assumption}

The preceding assumptions on the data generating process trivially nest i.i.d. sequences but also allow for quite general weakly-dependent data. In an i.i.d. setting, Assumption \ref{resid}(ii) reduces to requiring that $E[\epsilon_i^2 |X_i = x]$ be bounded uniformly from zero and infinity which is standard (see, e.g., \cite{Newey1997,HallHorowitz}). The value of $\delta$ in Assumption \ref{resid}(iii) depends on the context. For example, $\delta \geq d/p$ will be shown to be sufficient to attain the optimal sup-norm convergence rates for series LS regression in Section \ref{reg sec}, whereas lower values of $\delta$ suffice to attain the optimal sup-norm convergence rates for the sieve NPIV estimator in Section \ref{o-npiv sec}. Rectangular support and bounded densities of the endogenous regressor and instrument are assumed in \cite{HallHorowitz}. Assumptions \ref{f sieve}(i) and \ref{b sieve}(i) are satisfied by many widely used sieve bases such as spline, wavelet and cosine sieves, but they rule out polynomial and power series sieves (see, e.g., \cite{Newey1997,Huang1998}). The instruments sieve basis $b^K(\cdot) $ is used to approximate the conditional expectation operator $Th=E[h(Y_2)|X=\cdot)$, which is a smoothing operator. Thus Assumption \ref{b sieve}(i) assumes that the sieve basis $b^K(\cdot) $ (for $Th$) is smoother than that of the sieve basis $\psi^J(\cdot) $ (for $h$).

In the next theorem, our upper bound on the ``variance term'' $\|\widehat h - P_n h_0\|_{\infty}$ holds under general weak dependence as captured by Condition (ii) on the convergence of the random matrices $\widetilde B'\widetilde B/n - I_K$ and $\widehat S - S$.

\begin{theorem}\label{sup norm rate gen new}
Let Assumptions \ref{data}, \ref{resid}, \ref{f sieve} and \ref{b sieve} hold. If $\sigma_{JK} > 0$ then:
\begin{equation*}
 \|h_0 - \widehat h\|_{\infty} \leq \|h_0 - P_n h_0\|_{\infty} + O_p \left( \sigma_{JK}^{-1}\sqrt{K (\log n)/n}\right)
\end{equation*}
provided $n,J,K \to \infty$ and
\begin{enumerate}[(i)]
\item $J \leq K$, $K \lesssim (n/\log n)^{\delta/(2+\delta)}$, and $\sigma_{JK}^{-1} \sqrt{K (\log n)/n} \lesssim 1$
\item $\sigma_{JK}^{-1} \left( \|(\widetilde B'\widetilde B/n) - I_K\| +  \|\widehat S - S\|\right)= O_p(\sqrt {(\log n)/K})= o_p(1)$.
\end{enumerate}
\end{theorem}

The restrictions on $J$, $K$ and $n$ in Conditions (i) and (ii) merit a brief explanation. The restriction $J \leq K$ merely ensures that the sieve NPIV estimator is well defined. The restriction $K \lesssim (n/\log n)^{\delta/(2+\delta)}$ is used to perform a truncation argument using the existence of $(2+\delta)$-th moment of the error terms (see Assumption \ref{resid}). Condition (ii) ensures that $J$ increases sufficiently slowly that with probability approaching one the minimum eigenvalue of the ``denominator'' matrix $\Psi'B(B'B)^-B'\Psi/n$ is positive and bounded below by a multiple of $\sigma_{JK}^2$, thereby regularizing the ill-posed inverse problem. It also ensures the error in estimating the matrices $(\widetilde B'\widetilde B/n)$ and $\widehat S$ vanishes sufficiently quickly that it doesn't affect the convergence rate of the estimator.

\begin{remark}
Section \ref{ei sec} provides very mild low-level sufficient conditions for Condition (ii) to hold under weakly dependent data. In particular, when specializing Corollary \ref{troppcor} to i.i.d. data $\{(X_i,Y_{2i})\}_{i=1}^n$ (also see Lemma \ref{Bconvi.i.d.}), under Assumptions \ref{f sieve} and \ref{b sieve} and $J \leq K$, we have:
\begin{equation*}
 \| (\widetilde B'\widetilde B/n) - I_K\| =  O_p(\sqrt{K (\log K)/n}),~~
 \| \widehat S - S\|  =  O_p(\sqrt{K (\log K)/n}).
\end{equation*}
\end{remark}

\section{Optimal uniform convergence rates for sieve NPIV estimators}\label{o-npiv sec}

\subsection{Upper bounds on uniform convergence rates for sieve NPIV estimators}

We now exploit the specific linear structure of the sieve NPIV estimator to derive uniform convergence rates for the mildly and severely ill-posed cases. Some additional assumptions are required so as to control the ``bias term'' $\|h_0 - P_n h_0\|_{\infty}$ and to relate the estimator to the measure of ill-posedness.

\textbf{$p$-smooth H\"older class of functions.} We first impose a standard smoothness condition on the unknown structural function $h_0$ to facilitate comparison with \cite{Stone1982}'s minimax risk lower bound in sup-norm loss for a nonparametric regression function.
Recall that $\mathcal Y_2 = [0,1]^d$. Deferring definitions to \cite{Triebel2006,Triebel2008}, we let $B^{p}_{q,q}([0,1]^d)$ denote the Besov space of smoothness $p$ on the domain $[0,1]^d$ and $\|\cdot\|_{B^{p}_{q,q}}$ denote the usual Besov norm on this space. 
Special cases include the Sobolev class of smoothness $p$, namely $B^{p}_{2,2}([0,1]^d)$, and the H\"older-Zygmund class of smoothness $p$, namely $B^{p}_{\infty,\infty}([0,1]^d)$. Let $B(p,L)$ denote a H\"older ball of smoothness $p$ and radius $0 <L<\infty$, i.e. $B(p,L) = \{ h \in B^p_{\infty,\infty}([0,1]^d) : \|h\|_{B^p_{\infty,\infty}} \leq L\}$.

\begin{assumption}\label{parameter regression}
$h_0 \in \mathcal H = B^{p}_{\infty,\infty}([0,1]^d)$ for some $p \geq d/2$.
\end{assumption}

Assumptions \ref{f sieve} and \ref{parameter regression} imply that there is $\pi_J h_0 \in \Psi_J$ such that $\|h_0 - \pi_J h_0\|_{\infty} = O(J^{-p/d})$.

\textbf{Sieve measure of ill-posedness.} Let $T : L^q(Y_2) \to L^q(X)$ denote the conditional expectation operator for $1\leq q \leq \infty$:
\begin{equation}
 T h(x) = E[h(Y_{2i})|X_i = x]\,.
\end{equation}
When $Y_2$ is endogenous, $T$ is compact under mild conditions on the conditional density of $Y_2$ given $X$. For $q'\geq q\geq 1$, we define a measure of ill-posedness (over a sieve space $\Psi_J$) as
\begin{equation}
\tau_{q,q',J} = \sup_{h \in \Psi_J : \|Th\|_{L^q(X)} \neq 0} \frac{\|h\|_{L^{q'}(Y_2)}}{%
\|Th\|_{L^q(X)}}\,.
\end{equation}
The $\tau_{2,2,J}$ measure of ill-posedness is clearly related to our earlier definition of $\sigma_{JK}$. By definition
\begin{equation*}
 \sigma_{JK} = \inf_{h \in \Psi_J : \|h\|_{L^2(Y_2)} = 1} \|\Pi_K Th\|_{L^2(X)} \leq \inf_{h \in \Psi_J : \|h\|_{L^2(Y_2)} = 1} \| Th\|_{L^2(X)}
 =\left( \tau_{2,2,J} \right)^{-1}
\end{equation*}
when $J \leq K$. The sieve measures of ill-posedness, $\tau_{2,2,J}$ and $\sigma_{JK}^{-1}$, are clearly non-decreasing in $J$. In \citet*{Blundell2007}, \cite{Horowitz2011} and \cite{ChenPouzo2012}, the NPIV model is said to be

\emph{$\bullet$ mildly ill-posed} if $\tau_{2,2,J} = O(J^{\varsigma/d})$ for some $\varsigma >
0 $;

\emph{$\bullet$ severely ill-posed} if $\tau_{2,2,J} = O(\exp(\frac{1}{2}
J^{\varsigma /d}))$ for some $\varsigma > 0$.

These measures of ill-posedness are not exactly the same as (but are related to) the measure of ill-posedness used in \cite{HallHorowitz} and \cite{Cavalier2008}. In the latter papers, it is assumed that the compact operator $T : L^2(Y_2) \to L^2(X)$ admits a singular value decomposition $\{\mu_{k};\phi_{1k},\phi_{0k}\}_{k=1}^{\infty }$, where $\{\mu_{k}\}_{k=1}^{\infty }$
are the singular numbers arranged in non-increasing order ($\mu_{k}\geq \mu_{k+1}\searrow 0$), $\{\phi_{1k}(y_{2})\}_{k=1}^{\infty }$ and $\{\phi_{0k}(x)\}_{k=1}^{\infty }$ are eigenfunction (orthonormal) bases for $%
L^{2}(Y_{2})$ and $L^{2}(X)$ respectively, and ill-posedness is measured in terms of the rate of decay of the singular values towards zero. Denote $T^{\ast }$ as the adjoint operator of $T$: $\{T^{\ast }g\}(Y_{2})\equiv E[g(X)|Y_{2}]$, which maps $L^{2}(X)
$ into $L^{2}(Y_{2})$. Then a compact $T$ implies that $T^{\ast }$, $T^{\ast}T$ and $TT^{\ast}$ are also
compact, and that $T\phi_{1k}=\mu_{k}\phi_{0k}$ and $T^{\ast }\phi_{0k}=\mu_{k}\phi_{1k}$ for all $k$. We note that $\|Th \|_{L^2(X)}=\|(T^*T)^{1/2}h \|_{L^2(Y_2)}$ for all $h \in Dom(T)$.
The following lemma provides some relations between these different measures of ill-posedness.

\begin{lemma}\label{ill-suff}
Let the conditional expectation operator $T : L^2(Y_2) \to L^2(X)$ be compact and injective. Then: (1) $\sigma_{JK}^{-1}\geq \tau_{2,2,J}\geq 1/\mu_{J}$; (2) If the sieve space
$\Psi_{J}$ spans the closed linear subspace (in $L^{2}(Y_{2})$)
generated by $\left\{ \phi_{1k}:k=1,...,J\right\} $, then: $%
\tau_{2,2,J}\leq 1/\mu_{J}$; (3) If, in addition, $J \leq K$ and the sieve space $B_{K}$ contains the closed
linear subspace (in $L^{2}(X)$) generated by $\left\{ \phi_{0k}:k=1,...,J\right\}$, then: $\sigma_{JK}^{-1}\leq
1/\mu_{J}$ and hence $\sigma_{JK}^{-1} = \tau_{2,2,J} = 1/\mu_J$.
\end{lemma}

Lemma \ref{ill-suff} parts (1) and (2) is Lemma 1 of \citet*{Blundell2007}, while Lemma \ref{ill-suff} part (3) is proved in the Appendix. We next present a sufficient condition to bound the sieve measures of ill-posedness $\sigma_{JK}^{-1}$ and $ \tau_{2,2,J}$.

\begin{assumption}
\label{r-link-T} (\textbf{sieve reverse link condition}) There is a continuous increasing function $\varphi :%
\mathbb R_{+}\rightarrow \mathbb R_{+}$ such that: (a) $\|Th \|_{L^2(X)}^{2}\gtrsim
\sum_{j=1}^{J }\varphi (j^{-2/d})|E[h(Y_2)\widetilde \psi_{Jj} (Y_2)]|^{2} $ for all $h\in \Psi_{J}$; or (b) $\|\Pi_K Th \|_{L^2(X)}^{2}\gtrsim
\sum_{j=1}^{J}\varphi (j^{-2/d})|E[h(Y_2)\widetilde \psi_{Jj} (Y_2)]|^{2}$ for all $h\in \Psi_{J}$
\end{assumption}

It is clear that Assumption \ref{r-link-T}(b) implies Assumption \ref{r-link-T}(a). Assumption \ref{r-link-T}(a) is the so-called ``sieve reverse link condition'' used in \cite{ChenPouzo2012}, which is weaker than the ``reverse link condition'' imposed in \cite{ChenReiss} and others in the ill-posed inverse literature: $\|Th \|_{L^2(X)}^{2}\gtrsim
\sum_{j=1}^{\infty }\varphi (j^{-2/d})|E[h(Y_2)\widetilde \psi_{Jj} (Y_2)]|^{2} $ for all $h\in B(p,L)$. We immediately have the following bounds:

\begin{remark}\label{r-link-ill-suff}
(1) Assumption \ref{r-link-T}(a) implies that $\tau _{2,2,J}\lesssim \left(\varphi (J^{-2/d})\right)^{-1/2} $. (2) Assumption \ref{r-link-T}(b) implies that $\tau _{2,2,J} \leq \sigma_{JK}^{-1} \lesssim \left(\varphi (J^{-2/d})\right)^{-1/2}$.
\end{remark}
Given Remark \ref{r-link-ill-suff}, in this paper we could call a NPIV model

\emph{$\bullet$ mildly ill-posed} if $\sigma_{JK}^{-1} = O(J^{\varsigma/d})$ or $\varphi (t)=t^{\varsigma}$ for some $\varsigma >
0 $;

\emph{$\bullet$ severely ill-posed} if $\sigma_{JK}^{-1} = O(\exp(\frac{1}{2}
J^{\varsigma /d}))$ or $\varphi (t)=\exp(-t^{-\varsigma /2})$ for some $\varsigma > 0$.

Define
\begin{equation}
 \sigma_{\infty, JK} = \inf_{h \in \Psi_J : \|h\|_{\infty} = 1} \|\Pi_K Th \|_{\infty} \leq \left( \tau_{\infty,\infty,J} \right)^{-1}.
\end{equation}

\begin{assumption}\label{modulus}
(i) The conditional expectation operator $T : L^q(Y_2) \to L^q(X)$ is compact and injective for $q=2$ and $q=\infty$, (ii) $\sigma_{\infty,JK}^{-1}\|\Pi_K T (h_0 - \pi_J h_0)\|_{\infty} \lesssim \|h_0 - \pi_J h_0\|_{\infty}$.
\end{assumption}

Assumption \ref{modulus}(ii) is a sup-norm analogue of the so-called ``stability condition'' imposed in the ill-posed inverse regression literature, such as Assumption 6 of \citet*{Blundell2007} and Assumption 5.2(ii) of \cite{ChenPouzo2012}.

To control the ``bias term'' $\|P_n h_0 - h_0\|_{\infty}$, we will use spline or wavelet sieves in Assumptions \ref{f sieve} and \ref{b sieve} so that we can make use of sharp bounds on the approximation error due to \cite{Huang2003}.\footnote{The key property of spline and wavelet sieve spaces that permits this sharp bound is their local support (see the appendix to \cite{Huang2003}). Other sieve bases such as orthogonal polynomial bases do not have this property and are therefore unable to attain the optimal sup-norm convergence rates for NPIV or nonparametric series LS regression.}
Control of the ``bias term'' $\|P_n h_0 - h_0\|_{\infty}$ is more involved in the sieve NPIV context than the sieve nonparametric LS regression context. In particular, control of this term makes use of an additional argument using exponential inequalities. To simplify presentation, the next theorem just presents the uniform convergence rate for sieve NPIV estimators under i.i.d. data.

\begin{theorem}\label{sup norm rate npiv new}
Let Assumptions \ref{data}, \ref{resid}, \ref{f sieve} (with $\Psi_J = \mbox{BSpl}(J,[0,1]^d,\gamma ) ~or~ \mbox{Wav}(J,[0,1]^d, \gamma)$), \ref{b sieve} (with $B_K = \mbox{BSpl}(K,[0,1]^{d},{\gamma}_x )~or~ \mbox{Wav}(K,[0,1]^{d}, {\gamma}_x)$), \ref{parameter regression} and \ref{modulus} hold. If $\{(X_i,Y_{2i})\}_{i=1}^n$ is i.i.d. then:
\begin{equation*}
 \|h_0 - \widehat h \|_{\infty} = O_p ( J^{-p/d} + \sigma_{JK}^{-1} \sqrt{K (\log n)/n}  )
\end{equation*}
provided $J \leq K$, $K \lesssim (n/\log n)^{\delta/(2+\delta)}$, and $\sigma_{JK}^{-1} K\sqrt{(\log n)/ n} \lesssim 1$ as $n,J,K \to \infty$.

(1) Mildly ill-posed case ($\sigma_{JK}^{-1} = O(J^{\varsigma/d})$ or $\varphi (t)=t^{\varsigma}$). If Assumption \ref{resid} holds with $\delta \geq d/(\varsigma + p)$, and $ J \asymp K \asymp (n/\log n)^{d/(2(p+\varsigma)+d)}$ with $K/J \to c_0 \geq 1$, then:
\begin{equation*}
 \|h_0 - \widehat h \|_{\infty} = O_p ( (n/\log n)^{-p/(2(p+\varsigma)+d)}).
\end{equation*}

(2) Severely ill-posed case ($\sigma_{JK}^{-1} = O(\exp(\frac{1}{2}
J^{\varsigma /d}))$ or $\varphi (t)=\exp(-t^{-\varsigma /2})$). If Assumption \ref{resid} holds with $\delta >0$, and $J = c_0'(\log n)^{d/\varsigma}$ for any $c_0' \in (0,1)$ with $K = c_0 J$ for some finite $c_0 \geq 1$, then:
\begin{equation*}
 \|h_0 - \widehat h \|_{\infty} = O_p ( (\log n)^{-p/\varsigma}).
\end{equation*}
\end{theorem}

\begin{remark}\label{L2 norm rates}
Under conditions similar to those for Theorem \ref{sup norm rate npiv new}, \citet*{Blundell2007}, \cite{ChenReiss} and \cite{ChenPouzo2012} previously obtained the following $L^2(Y_2) $-norm convergence rate for the sieve NPIV estimator:
\begin{equation*}
 \|h_0 - \widehat h \|_{L^2(Y_2)} = O_p ( J^{-p/d} + \tau_{2,2,J} \sqrt{K /n}  ).
\end{equation*}

(1) Mildly ill-posed case ($\tau_{2,2,J} = O(J^{\varsigma/d})$ or $\varphi (t)=t^{\varsigma}$),
\begin{equation*} \label{l2 rate mild}
 \|h_0 - \widehat h \|_{L^2(Y_2)} = O_p(n^{-p/(2(p+\varsigma)+d)})\,.
\end{equation*}

(2) Severely ill-posed case ($\tau_{2,2,J} = O(\exp(\frac{1}{2}
J^{\varsigma /d}))$ or $\varphi (t)=\exp(-t^{-\varsigma /2})$),
\begin{equation*} \label{l2 rate severe}
 \|h_0 - \widehat h\|_{L^2(Y_2)} = O_p((\log n)^{-p/\varsigma})\,.
\end{equation*}
\end{remark}
\cite{ChenReiss} show that these $L^2(Y_2) $-norm rates are optimal in the sense that they coincide with the minimax risk lower bound in $L^2(Y_2) $ loss. It is interesting to see that our sup-norm convergence rate is the same as the known optimal $L^2(Y_2) $-norm rate for the severely ill-posed case, and is only power of $\log(n)$ slower than the known optimal $L^2(Y_2) $-norm rate for the mildly ill-posed case. In the next subsection we will show that our sup-norm convergence rates are in fact optimal as well.

\subsection{Lower bounds on uniform convergence rates for NPIR and NPIV models}

For severely ill-posed NPIV models, \cite{ChenReiss} already showed that $(\log n)^{-p/\varsigma}$ is the minimax lower bound in $L^2(Y_2)$-norm loss uniformly over a class of functions that include the H\"older ball $B(p,L)$ as a subset.
Therefore, we have for a severely ill-posed NPIV model with $\delta_n = (\log n)^{-p/\varsigma}$,
\begin{equation*}
 \inf_{\widetilde h_n} \sup_{h \in B(p,L)} \mathbb P_h \left( \|h - \widetilde h_n \|_\infty \geq c\delta_n  \right) \geq
   \inf_{\widetilde h_n} \sup_{h \in B(p,L)} \mathbb P_h \left( \|h - \widetilde h_n \|_{L^2(Y_2)} \geq c\delta_n \right)
 \geq c'
\end{equation*}
where $\inf_{\widetilde h_n}$ denotes the infimum over all estimators based on a random sample of size $n$ drawn from the NPIV model, and the finite positive constants $c, c'$ do not depend on sample size $n$. This and Remark \ref{L2 norm rates}(2) together imply that the sieve NPIV estimator attains the optimal uniform convergence rate in the severely ill-posed case.

We next show that the sup-norm rate for the sieve NPIV estimator obtained in the mildly ill-posed case is also optimal. We begin by placing a primitive smoothness condition on the conditional expectation operator $T : L^2(Y_2) \to L^2(X)$.

\begin{assumption} \label{smoothness}
There is a $\varsigma > 0$ such that $\|Th \|_{L^2(X)} \lesssim \| h\|_{B^{-\varsigma}_{2,2}}$ for all $h \in B(p,L)$.
\end{assumption}

Assumption \ref{smoothness} is a special case of the so-called ``link condition'' in \cite{ChenReiss} for the mildly ill-posed case. It can be equivalently stated as: $\|Th \|_{L^2(X)}^{2} \lesssim
\sum_{j=1}^{\infty }\varphi (j^{-2/d})|E[h(Y_2)\widetilde \psi_{Jj} (Y_2)]|^{2} $ for all $h\in B(p,L)$, with $\varphi (t)=t^{\varsigma}$ for the mildly ill-posed case. Under this assumption, $n^{-p/(2(p+\varsigma)+d)}$ is the minimax risk lower bound uniformly over the H\"older ball $B(p,L)$ in $L^2 (Y_2)$-norm loss for the mildly ill-posed NPIR and NPIV models (see \cite{ChenReiss}). We next establish the corresponding minimax risk lower bound in sup-norm loss.

\begin{theorem} \label{npiv lower bound}
Let Assumption \ref{smoothness} hold for the NPIV model with a random sample $\{(Y_{1i},Y_{2i},X_i)\}_{i=1}^n$. Then:
\begin{equation*}
 \liminf_{n \to \infty} \inf_{\widetilde h_n} \sup_{h \in B(p,L)} \mathbb P_h \left( \|h - \widetilde h_n \|_\infty \geq c(n/\log n)^{-p/(2(p+\varsigma)+d)} \right) \geq c'>0,
\end{equation*}
where $\inf_{\widetilde h_n}$ denotes the infimum over all estimators based on the sample of size $n$, and the finite positive constants $c, c'$ do not depend on $n$.
\end{theorem}

As in \cite{ChenReiss}, Theorem \ref{npiv lower bound} is proved by (i) noting that the risk (in sup-norm loss) for the NPIV model is at least as large as the risk (in sup-norm loss) for the NPIR model, and (ii) calculating a lower bound (in sup-norm loss) for the NPIR model. We consider a Gaussian reduced-form NPIR model with known operator $T$, given by
\begin{equation}  \label{npir}
\begin{array}{rcl}
Y_{1i} & = & Th_0(X_{i}) + u_i, \quad i=1,...,n,\\
u_i|X_i & \sim & N(0,\sigma^2 (X_i) )\quad \mbox{with} \quad \inf_x \sigma^2 (x)\geq \sigma_0^2 >0\,.
\end{array}%
\end{equation}
Theorem \ref{npiv lower bound} therefore follows from a sup-norm analogue of Lemma 1 of \cite{ChenReiss} and the following theorem, which establishes a lower bound on minimax risk over H\"older classes under sup-norm loss for the NPIR model.

\begin{theorem}\label{npir lower bound}
Let Assumption \ref{smoothness} hold for the NPIR model (\ref{npir}) with a random sample $\{(Y_{1i},X_i)\}_{i=1}^n$. Then:
\begin{equation*}
 \liminf_{n \to \infty} \inf_{\widetilde h_n} \sup_{h \in B(p,L)} \mathbb P_h \left( \|h - \widetilde h_n \|_\infty \geq c (n/\log n)^{-p/(2(p+\varsigma)+d)} \right) \geq c'>0,
\end{equation*}
where $\inf_{\widetilde h_n}$ denotes the infimum over all estimators based on the sample of size $n$, and the finite positive constants $c, c'$ depend only on $p,L,d,\varsigma$ and $\sigma_0$.
\end{theorem}

\section{Optimal uniform convergence rates for sieve LS estimators} \label{reg sec}

The standard nonparametric regression model can be recovered as a special case of (\ref{npiv}) in which there is no endogeneity, i.e. $Y_2 = X$ and
\begin{equation}  \label{npreg}
\begin{array}{rcl}
Y_{1i} & = & h_0(X_i) + \epsilon_i \\
E[\epsilon_i |X_i] & = & 0%
\end{array}%
\end{equation}
in which case $h_0(x) = E[Y_{1i}|X_i = x]$.

\cite{Stone1982} (also see \cite{Tsybakov2009}) establishes that $(n/\log n)^{-p/(2p +d)}$ is the minimax risk lower bound in sup-norm loss for the nonparametric LS regression model (\ref{npreg}) with $h_0 \in B(p,L)$. In this section we apply the general upper bound (Theorem \ref{sup norm rate gen new}) to show that spline and wavelet sieve LS estimators attain this minimax lower bound for weakly dependent data allowing for heavy-tailed error terms $\epsilon_i$.

Our proof proceeds by noticing that the sieve LS regression estimator
\begin{equation}
 \widehat h(x) = b^K(x)(B'B)^-B'Y
\end{equation}
obtains as a special case of the NPIV estimator by setting $Y_2 = X$, $\psi^J = b^K$, $J = K$ and $\gamma = \gamma_x$. In this setting, the quantity $P_n h_0(x)$ just reduces to the orthogonal projection of $h_0$ onto the sieve space $B_K$ under the inner product induced by the empirical distribution, viz.
\begin{equation}
 P_n h_0(x)=\widetilde b^K(x)(\widetilde B'\widetilde B/n)^{-} \widetilde B' H_0/n\,.
\end{equation}
Moreover, in this case the $J \times K$ matrix $S$ defined in (\ref{S def}) reduces to the $K \times K$ identity matrix $I_K$ and its smallest singular value is unity (whence $\sigma_{JK} = 1$). Therefore, the general calculation presented in Theorem \ref{sup norm rate gen new} can be used to control the ``variance term'' $\|\widehat h - P_n h_0\|_{\infty}$. The ``bias term'' $\|P_n h_0 - h_0\|_{\infty}$ is controlled as in \cite{Huang2003}. It is worth emphasizing that no explicit weak dependence condition is placed on the regressors $\{X_i\}_{i=-\infty}^\infty$. Instead, this is implicitly captured by Condition (ii) on convergence of $\widetilde B'\widetilde B/n - I_K$.

\begin{theorem}\label{sup norm rate regression}
Let Assumptions \ref{data}, \ref{resid}, \ref{b sieve} (with $B_K = \mbox{BSpl}(K,[0,1]^{d},{\gamma} )~or~ \mbox{Wav}(K,[0,1]^{d}, {\gamma})$) and \ref{parameter regression} hold for Model (\ref{npreg}). Then:
\begin{equation*}
 \|\widehat h - h_0\|_{\infty} = O_p ( K^{-p/d} + \sqrt{K(\log n)/n} )
\end{equation*}
provided $n,K \to \infty$, and
\begin{enumerate}[(i)]
\item $K \lesssim (n/\log n)^{\delta/(2+\delta)}$ and $\sqrt{K(\log n)/n} \lesssim 1$
\item $ \|(\widetilde B'\widetilde B/n) - I_K\|= O_p(\sqrt{(\log n) /K}) = o_p(1)$.
\end{enumerate}
\end{theorem}

Condition (ii) is satisfied by applying Lemma \ref{Bconvi.i.d.} for i.i.d. data and Lemma \ref{Bconvbeta} for weakly dependent data. 
Theorem \ref{sup norm rate regression} shows that spline and wavelet sieve LS estimators can achieve this minimax lower bound for weakly dependent data.

\begin{corollary}\label{regcor}
Let Assumptions \ref{data}, \ref{resid} (with $\delta \geq d/p$), \ref{b sieve} (with $B_K = \mbox{BSpl}(K,[0,1]^{d},{\gamma} )~or~ \mbox{Wav}(K,[0,1]^{d}, {\gamma})$) and \ref{parameter regression} hold for Model (\ref{npreg}). If  $K \asymp (n/\log n)^{d/(2p +d)}$ then:
\begin{equation*}
 \|\widehat h - h_0\|_{\infty} = O_p ( (n/\log n)^{-p/(2p +d)} )
\end{equation*}
provided that one of the followings is satisfied
\begin{enumerate}[(1)]
\item the regressors are i.i.d.;
\item the regressors are exponentially $\beta$-mixing and $d < 2p$;
\item the regressors are algebraically $\beta$-mixing at rate $\gamma$ and  $(2+\gamma)d < 2 \gamma p$.
\end{enumerate}
\end{corollary}

Corollary \ref{regcor} states that for i.i.d. data, Stone's optimal sup-norm convergence rate is achieved by spline and wavelet LS estimators whenever $\delta \geq d/p$ and $d \leq 2p$ (Assumption \ref{parameter regression}). If the regressors are exponentially $\beta$-mixing the optimal rate of convergence is achieved with $\delta \geq d/p$ and $d < 2p$. The restrictions $\delta \geq d/p$ and $(2+\gamma)d < 2 \gamma p$ for algebraically $\beta$-mixing (at a rate $\gamma$) reduces naturally towards the exponentially mixing conditions as the dependence becomes weaker (i.e. $\gamma$ becomes larger). In all cases, a smoother function (i.e., bigger $p$) means a lower value of $\delta$, and therefore heaver-tailed error terms $\epsilon_i$, are permitted while still obtaining the optimal sup-norm convergence rate. In particular this is achieved with $\delta =d/p \leq 2$ for i.i.d. data. Recently, \cite{BCK2013} require that the conditional $(2+\eta)$th moment (for some $\eta >0$) of $\epsilon_i$ be uniformly bounded for spline LS regression estimators to achieve the optimal sup-norm rate for i.i.d. data.\footnote{Chen would like to thank Jianhua Huang for working together on an earlier draft that does achieve the optimal sup-norm rate for a polynomial spline LS estimator with i.i.d. data, but under a stronger condition that $E[\epsilon^4_i |X_i=x]$ is uniformly bounded in $x$.}
Uniform convergence rates of series LS estimators have also been studied by \cite{Newey1997}, \cite{deJong2002}, \cite{Song2008}, \cite{LeeRobinson} and others, but the sup-norm rates obtained in these papers are slower than the minimax risk lower bound in sup-norm loss of \cite{Stone1982}.\footnote{See, e.g., \cite{hansen-ET}, \cite{masry}, \cite{CattaneoFarrell} and the references therein for the optimal sup-norm convergence rates of a conditional mean function via the kernel, local linear regression and partitioning estimators of a conditional mean function.}
Our result is the first such optimal sup-norm rate result for a sieve nonparametric LS estimator allowing for weakly-dependent data with heavy-tailed error terms. It should be very useful for nonparametric estimation of financial time-series models that have heavy-tailed error terms.

\section{Useful results on random matrices}\label{ei sec}

\subsection{Convergence rates for sums of dependent random matrices}\label{ineq sec}

In this subsection a Bernstein inequality for sums of independent random matrices due to \cite{Tropp2012} is adapted to obtain convergence rates for sums of random matrices formed from $\beta$-mixing (absolutely regular) sequences, where the dimension, norm, and variance measure of the random matrices are allowed to grow with the sample size. These inequalities are particularly useful for establishing convergence rates for semi/nonparametric sieve estimators with weakly-dependent data. We first recall a result of \cite{Tropp2012}.

\begin{theorem}[\cite{Tropp2012}]\label{troppthm}
Let $\{\Xi_i\}_{i=1}^n$ be a finite sequence of independent random matrices with dimensions $d_1 \times d_2$. Assume $E[\Xi_i] = 0$ for each $i$ and $\max_{1 \leq i \leq n} \|\Xi_i\| \leq R_n$, and define
\begin{equation} \notag
 \sigma^2_n = \max\left\{ \left\| \sum_{i=1}^n E[\Xi_i\Xi_i'] \right\|, \left\| \sum_{i=1}^n E[\Xi_i'\Xi_i] \right\| \right\} \,.
\end{equation}
Then for all $t \geq 0$,
\begin{equation} \notag
 \mathbb P \left( \left\| \sum_{i=1}^n \Xi_i \right\|  \geq t \right) \leq (d_1 + d_2) \exp \left( \frac{-t^2/2}{\sigma_n^2 + R_n t/3} \right)\,.
\end{equation}
\end{theorem}

\begin{corollary}\label{troppcor}
Under the conditions of Theorem \ref{troppthm}, if $R_n \sqrt{\log(d_1+d_2)} = o(\sigma_n)$ then
\begin{equation} \notag
 \left\| \sum_{i=1}^n \Xi_{i,n} \right\| = O_p ( \sigma_n \sqrt{\log(d_1+d_2)} )\,.
\end{equation}
\end{corollary}

We now provide a version of Theorem \ref{troppthm} and Corollary \ref{troppcor} for matrix-valued functions of $\beta$-mixing  sequences.  The $\beta$-mixing coefficient between two $\sigma$-algebras $%
\mathcal{A}$ and $\mathcal{B}$ is defined as
\begin{equation}
2\beta(\mathcal{A},\mathcal{B}) = \sup \sum_{(i,j) \in I \times J} |\mathbb{P}%
(A_i \cap B_j) - \mathbb{P}(A_i) \mathbb{P}(B_j)|
\end{equation}
with the supremum taken over all finite partitions $\{A_i\}_{i \in I}\subset
\mathcal{A}$ and $\{B_J\}_{j \in J} \subset \mathcal{B}$ %
\citep*{DoukhanMassartRio}. The $q$th $\beta$-mixing coefficient of $%
\{X_i\}_{i=-\infty}^\infty$ is defined as
\begin{equation}
\beta(q) = \sup_i
\beta(\sigma(\ldots,X_{i-1},X_i),\sigma(X_{i+q},X_{i+q+1},\ldots))\,.
\end{equation}
The process $\{X_i\}_{i=-\infty}^\infty$ is said to be \emph{algebraically $%
\beta$-mixing} at rate $\gamma$ if $q^\gamma \beta(q) = o(1)$ for some $%
\gamma > 1$, and \emph{geometrically $\beta$-mixing} if $\beta(q) \leq c
\exp(-\gamma q)$ for some $\gamma > 0$ and $c \geq 0$. The following extension of Theorem \ref{troppthm} is made using a Berbee's lemma and a coupling argument (see, e.g., \cite{DoukhanMassartRio}).

\begin{theorem}\label{beta tropp}
Let $\{X_i\}_{i=-\infty}^\infty$ be a strictly stationary $\beta$-mixing sequence and let $\Xi_{i,n} = \Xi_n(X_i)$ for each $i$ where $\Xi_n : \mathcal X \to \mathbb R^{d_1 \times d_2}$ is a sequence of measurable $d_1 \times d_2$ matrix-valued functions. Assume $E[\Xi_{i,n}] = 0$ and $\|\Xi_{i,n}\| \leq R_n$ for each $i$ and define $s_n^2 = \max_{1 \leq i,j\leq n} \max\{ \| E[\Xi_{i,n}\Xi_{j,n}'] \|, \| E[\Xi_{i,n}'\Xi_{j,n}] \| \}$. Let $q$ be an integer between $1$ and $n/2$ and let $I_r =
q[n/q]+1,\ldots,n$ when $q[n/q] < n$ and $I_r = \emptyset$ when $q[n/q] = n$.
Then for all $t \geq 0$,
\begin{equation} \notag
 \mathbb P \left( \left\| \sum_{i=1}^n \Xi_{i,n} \right\|  \geq 6t \right) \leq \frac{n}{q} \beta (q) + \mathbb P \left( \left\|\sum_{i \in I_r} \Xi_{i,n} \right\| \geq t \right) + 2(d_1 + d_2) \exp \left( \frac{-t^2/2}{nq s_n^2 + qR_n t/3} \right)
\end{equation}
(where $\|\sum_{i \in I_r} \Xi_{i,n} \|:=0$ whenever $I_r = \emptyset$).
\end{theorem}

\begin{corollary}\label{beta rate}
Under the conditions of Theorem \ref{beta tropp}, if $q = q(n)$ is chosen such that $\frac{n}{q}\beta(q) = o(1)$ and $R_n \sqrt{q \log(d_1+d_2)} = o(s_n\sqrt{n})$ then
\begin{equation} \notag
 \left\| \sum_{i=1}^n \Xi_{i,n} \right\| = O_p ( s_n \sqrt{n q \log(d_1 + d_2)})\,.
\end{equation}
\end{corollary}

\subsection{Empirical identifiability}\label{ei sec-1}

This subsection provides a readily verifiable condition under which, with probability approaching one (wpa1), the theoretical and empirical $L^2$ norms are equivalent over a linear sieve space. This equivalence, referred to by \cite{Huang2003} as \emph{empirical identifiability}, has several applications in nonparametric sieve estimation. In the context of nonparametric series regression,
empirical identifiability ensures the estimator is
the orthogonal projection of $Y$ onto the sieve space under the empirical
inner product and is uniquely defined \citep{Huang2003}. Empirical
identifiability is also used to establish the
large-sample properties of sieve conditional moment estimators \citep{ChenPouzo2012}. A sufficient condition for empirical identifiability is now cast in terms of convergence of a random matrix, which we verify for i.i.d. and $\beta$-mixing sequences.

A subspace $\mathcal{A }\subseteq
L^2(X)$ is said to be \textit{empirically identifiable} if $\frac{1}{n}%
\sum_{i=1}^n b(X_i)^2 = 0$ implies $b = 0$ a.e.-$[F_X]$ where $F_X$ dentoes
the distribution of $X$. A sequence of spaces $\{\mathcal A_K : K \geq 1\} \subseteq L^2(X)$ is empirically identifiable wpa1 as $K = K(n) \to \infty$ with $n$ if
\begin{equation}  \label{sm5}
\lim_{n \to \infty} \mathbb{P }\left( \sup_{a \in A_K} \left| \frac{\frac{1}{%
n} \sum_{i=1}^n a(X_i)^2 - E[a(X)^2]}{E[a(X)^2]} \right| > t \right) = 0
\end{equation}
for any $t > 0$. \cite{Huang1998} uses a chaining argument to provide sufficient conditions for (\ref{sm5}) over the linear space $B_K$ under i.i.d. sampling. \cite{ChenPouzo2012} use this argument to establish convergence of sieve conditional moment estimators. Although easy to establish for i.i.d. sequences, it may be difficult to verify (\ref{sm5}) via chaining arguments for certain types of weakly dependent sequences. To this end, the following is a readily verifiable sufficient condition for empirical
identifiability for linear sieve spaces. Let $B_K = clsp\{b_{K1},\ldots,b_{KK}\}$ denote a general linear sieve space and let $\widetilde B = (\widetilde b^K(X_1),\ldots,\widetilde b^K(X_n))'$ where $\widetilde b^K(x)$ is the orthonormalized vector of basis functions.

\begin{condition} \label{ei cond}
$\lambda_{\min}(E[b^K(X)b^K(X)']) > 0$ for each $K \geq 1$ and $\|\widetilde B^{\prime }\widetilde B/n - I_K\| = o_p(1)$.
\end{condition}

\begin{lemma}
\label{eilem} If $\lambda_{\min}(E[b^K(X)b^K(X)']) > 0$ for each $K \geq 1$ then
\begin{equation} \notag
\sup_{b \in B_K} \left| \frac{\frac{1}{n} \sum_{i=1}^n b(X_i)^2 - E[b(X)^2]}{%
E[b(X)^2]} \right| = \|\widetilde B^{\prime }\widetilde B/n - I_K\|^2\,.
\end{equation}
\end{lemma}

\begin{corollary}
Under Condition \ref{ei cond}, $B_K$ is empirically identifiable wpa1.
\end{corollary}

Condition \ref{ei cond} is a sufficient condition for (\ref{sm5}) with a
linear sieve space $B_K$.
It should be noted that convergence is only required in the spectral norm. In the i.i.d. case this allows for $K$ to increase more quickly with $n$ than is achievable under the chaining argument of \cite{Huang1998}. Let
\begin{equation}
 \zeta_0(K) = \sup_{x \in \mathcal X}\|b^K(x)\|
\end{equation}
as in \cite{Newey1997}. Under regularity conditions, $\zeta_0(K) = O(\sqrt K)$ for tensor products
of splines, trigonometric polynomials or wavelets and $\zeta_0(K) = O(K)$
for tensor products of power series or polynomials %
\citep{Newey1997,Huang1998}. Under the chaining argument of \cite{Huang1998}, (\ref{sm5}) is achieved under the restriction $\zeta_0(K)^2K/n = o(1)$. \cite{Huang2003} relaxes this restriction to $K (\log n)/n = o(1)$ for a polynomial spline sieve. We now generalize this result by virtue of Lemma \ref{eilem} and exponential inequalities for sums of random matrices.

\begin{lemma}
\label{Bconvi.i.d.} If $\{X_i\}_{i=1}^n$ is i.i.d. and $%
\lambda_{\min}(E[b^K(X)b^K(X)']) \geq \underline \lambda > 0$ for each $K \geq 1$, then
\begin{equation} \notag
\|(\widetilde B^{\prime }\widetilde B/n) - I_K\| = O_p( \zeta_0(K)\sqrt{%
{(\log K)}/{n}})
\end{equation}
provided $\zeta_0(K)^2(\log K)/n = o(1)$.
\end{lemma}

\begin{remark}
If $\{X_i\}_{i=1}^n$ is i.i.d., $K (\log K)/n = o(1)$ is sufficient for sieve bases that are   tensor products
of splines, trigonometric polynomials or wavelets, and $K^2(\log K)/n = o(1)$ is sufficient for sieve bases that are tensor products of power series or polynomials.
\end{remark}

The following lemma is useful to provide sufficient conditions for empirical identifiability for $\beta$-mixing sequences, which uses Theorem \ref{beta rate}.

\begin{lemma}
\label{Bconvbeta} If $\{X_i\}_{i=-\infty}^\infty$ is strictly stationary and
$\beta$-mixing with mixing coefficients such that one can choose an integer
sequence $q = q(n) \leq n/2$ with $\beta(q) n/q = o(1)$ and $%
\lambda_{\min}(E[b^K(X)b^K(X)']) \geq \underline \lambda > 0$ for each $K \geq 1$, then
\begin{equation} \notag
\|(\widetilde B^{\prime }\widetilde B/n) - I_K\| = O_p(\zeta_0(K) \sqrt{%
{q (\log K)}/{n}})
\end{equation}
provided $\zeta_0(K)^2 q \log K/n = o(1)$.
\end{lemma}

\begin{remark}
If $\{X_i\}_{i=-\infty}^\infty$ is algebraically $\beta$-mixing at rate $%
\gamma$,  $K n^{1/(1+\gamma)} (\log K)/n = o(1)$ is sufficient for sieve bases that are   tensor products
of splines, trigonometric polynomials or wavelets, and $K^2 n^{1/(1+\gamma)} (\log K)/n = o(1)$ is sufficient for sieve bases that are tensor products of power series or polynomials.
\end{remark}

\begin{remark}
If $\{X_i\}_{i=-\infty}^\infty$ is geometrically $\beta$-mixing,  $K (\log n)^2/n = o(1)$ is sufficient for sieve bases that are   tensor products
of splines, trigonometric polynomials or wavelets, and $K^2 (\log n)^2/n = o(1)$ is sufficient for sieve bases that are tensor products of power series or polynomials.
\end{remark}

\appendix

\fontsize{10}{12}\selectfont

\section{Brief review of B-spline and wavelet sieve spaces} \label{sieve def}

We first outline univariate B-spline and wavelet sieve spaces on $[0,1]$, then deal with the multivariate case by constructing a tensor-product sieve basis.

\paragraph{B-splines}

B-splines are defined by their order $m \geq 1$ and number of interior knots $N \geq 0$. Define the knot set
\begin{equation}
 t_{-(m-1)} = \ldots = t_0 \leq t_1 \leq \ldots \leq t_{N} \leq t_{N+1} = \ldots = t_{N+m}
\end{equation}
where we normalize $t_0 = 0$ and $t_{N+1} = 1$. The B-spline basis is then defined recursively via the De Boor relation. This results in a total of $K = N+m$ splines which together form a partition of unity. Each spline is a polynomial of degree $m-1$ on each interior interval $I_1 = [t_0,t_1),\ldots,I_n = [t_N,t_{N+1}]$ and is $(m-2)$-times continuously differentiable on $[0,1]$ whenever $m \geq 2$. The mesh ratio is defined as
\begin{equation}
 \mbox{mesh}(K) = \frac{\max_{0 \leq n \leq N} (t_{n+1} - t_n)}{\min_{0 \leq n \leq N} (t_{n+1} - t_n)}\,.
\end{equation}
We let the space $\mbox{BSpl}(K,[0,1])$ be the closed linear span of these $K = N+m$ splines. The space $\mbox{BSpl}(K,[0,1])$ has \emph{uniformly bounded mesh ratio} if $\mbox{mesh}(K) \leq \kappa$ for all $N \geq 0$ and some $\kappa \in (0, \infty)$. The space $\mbox{BSpl}(K,[0,1])$ has \emph{smoothness} $\gamma= m- 2$, which is denoted as $\mbox{BSpl}(K,[0,1], \gamma )$ for simplicity. See \cite{deBoor2001} and \cite{Schumacker2007} for further details.

\paragraph{Wavelets}

We follow the construction of \cite{CDJV1993,CDV1993} for building a wavelet basis for $[0,1]$. Let $(\phi,\psi)$ be a father and mother wavelet pair that has $N$ vanishing moments and $\mbox{support}(\phi) = \mbox{support}(\psi) = [0,2N-1]$. For given $j$, the approximation space $V_j$ wavelet space $W_j$ each consist of $2^j$ functions $\{\phi_{jk}\}_{1 \leq k \leq 2^j}$ and $\{\psi_{jk}\}_{1 \leq k \leq 2^j}$ respectively, such that $\{\phi_{jk}\}_{1 \leq k \leq 2^j-2N}$ and $\{\psi_{jk}\}_{1 \leq k \leq 2^j-2N}$ are interior wavelets for which $\phi_{jk}(\cdot) = 2^{j/2}\phi(2^j(\cdot)-k)$ and $\psi_{jk}(\cdot) = 2^{j/2}\psi(2^j(\cdot)-k)$, complemented with another $N$ left-edge functions and $N$ right-edge functions. Choosing $L \geq 1$ such that $2^L \geq 2N$, we let the space $\mbox{Wav}(K,[0,1])$ be the closed linear span of the set of functions
\begin{equation}
 W_{LJ} = \{ \phi_{Lk} : 1 \leq k \leq 2^L\} \cup \{ \psi_{jk} : k = 1,\ldots,2^j \mbox{ and } j = L,\ldots,J-1\}
\end{equation}
for integer $J > L$, and let $K = \#(W_{LJ})$. We say that $\mbox{Wav}(K,[0,1])$ has \emph{regularity} $\gamma$ if $\phi$ and $\psi$ are both $\gamma$ times continuously differentiable, which is denoted as $\mbox{Wav}(K,[0,1], \gamma )$ for simplicity.

\paragraph{Tensor products} To construct a tensor-product B-spline basis of smoothness $\gamma$ for $[0,1]^d$ with $d > 1$, we first construct $d$ univariate B-spline bases for $[0,1]$, say $G_i$ with $G_i = \mbox{BSpl}(k,[0,1])$ and smoothness $\gamma$ for each $1 \leq i \leq d$. We then set $K = k^d$ and let $\mbox{BSpl}(K,[0,1]^d)$ be spanned by the unique $k^d$ functions given by $\prod_{i=1}^d g_i$ with $g_i \in G_i$ for $1 \leq i \leq d$. The tensor-product wavelet basis $\mbox{Wav}(K,[0,1]^d)$ of regularity $\gamma$ for $[0,1]^d$ is formed similarly as the tensor product of $d$ univariate Wavelet bases of regularity $\gamma$ (see \cite{Triebel2006,Triebel2008}).

\paragraph{Wavelet characterization of Besov norms} Let $f \in B^\alpha_{p,q}([0,1]^d)$ have wavelet expansion
\begin{equation}
 f = \sum_{k=-\infty}^\infty \mathsf{a}_{k}(f) \phi_{Lk} + \sum_{j=L}^\infty \sum_{k=-\infty}^\infty \mathsf{b}_{jk}(f) \psi_{jk}
\end{equation}
where $\{\phi_{Lk},\psi_{jk}\}_{j,k}$ are a Wavelet basis with regularity $\gamma > \alpha$. Equivalent norms to the $B^\alpha_{\infty,\infty}$ and $B^\alpha_{2,2}$ norms may be formulated equivalently in terms of the wavelet coefficient sequences $\{\mathsf{a}_k\}_k^\infty$ and $\{\mathsf{b}_{jk}\}_{j,k}$, namely $\|\cdot\|_{b^\alpha_{\infty,\infty}}$ and $\|\cdot\|_{b^\alpha_{2,2}}$, given by
\begin{equation} \begin{array}{rcl}
 \|f\|_{b^\alpha_{\infty,\infty}} & = & \sup_k |\mathsf{a}_{k}(f)| + \sup_{j,k} 2^{j(\alpha+d/2)}|\mathsf{b}_{jk}(f)| \\
 \|f\|_{b^\alpha_{2,2}} & = & \|\mathsf{a}_{(\cdot)}(f)\| + \left( \sum_{j=0}^\infty (2^{j\alpha}\|\mathsf{b}_{j(\cdot)}(f)\|)^2\right)^{1/2}
 \end{array}
\end{equation}
where $\|\mathsf{a}_{(\cdot)}(f)\|$ and $\|\mathsf{b}_{j(\cdot)}(f)\|$ denote the infinite-dimensional Euclidean norm for the sequences $\{\mathsf{a}_{k}(f)\}_k$ and $\{\mathsf{b}_{jk}(f)\}_{k}$ (see, e.g., \cite{Johnstone2013} and \cite{Triebel2006,Triebel2008}).

\section{Proofs of main results}

\subsection{Proofs for Section \ref{npiv sec}}

\begin{proof}[Proof of Theorem \ref{sup norm rate gen new}]
It is enough to show that $\|\widehat h - P_n h_0\|_{\infty} = O_p(\sigma_{JK}^{-1}\sqrt{K (\log n)/n})$.
First write
\begin{equation}
 \widehat h(y_2) - P_n h_0(y_2)  = \widetilde \psi^J(y_2)' [\widehat S(\widetilde B'\widetilde B/n)^{-} \widehat S']^{-} \widehat S (\widetilde B'\widetilde B/n)^{-} \widetilde B'e/n
\end{equation}
where $e = (\epsilon_1,\ldots,\epsilon_n)$.
Convexity of $\mathcal Y_2$ (Assumption \ref{data}(ii)), smoothness of $\widetilde \psi^J$ (Assumption \ref{f sieve}(i)) and the mean value theorem provide that, for any $(y,y^*) \in \mathcal Y_2^2$,
\begin{eqnarray}
 |\widehat h(y) - P_n h_0(y) - (\widehat h(y^*) - P_n h_0(y^*))| & = & |(\widetilde \psi^J(y) - \widetilde \psi^J(y^{*}))' [\widehat S(\widetilde B'\widetilde B/n)^{-} \widehat S']^{-} \widehat S (\widetilde B'\widetilde B/n)^{-} \widetilde B'e/n|  \\
 & = & |(y - y^{*})'\nabla \widetilde \psi^J(y^{**})'[\widehat S(\widetilde B'\widetilde B/n)^{-} \widehat S']^{-} \widehat S (\widetilde B'\widetilde B/n)^{-} \widetilde B'e/n| \\
 & \leq & J^\alpha \|y - y^{*}\| \|[\widehat S(\widetilde B'\widetilde B/n)^{-} \widehat S']^{-} \widehat S (\widetilde B'\widetilde B/n)^{-1} \widetilde B'e/n\|
\end{eqnarray}
for some $y^{**}$ in the segment between $y$ and $y^{*}$, and some $\alpha > 0$ (and independent of $y$ and $y^{*}$).

We first show that $T_1 := \|[\widehat S(\widetilde B'\widetilde B/n)^{-} \widehat S']^{-} \widehat S (\widetilde B'\widetilde B/n)^{-} \widetilde B'e/n\| = o_p(1)$. By the triangle inequality and properties of the matrix spectral norm,
\begin{eqnarray}
 T_1 & \leq & (\|[\widehat S(\widetilde B'\widetilde B/n)^{-} \widehat S']^{-} \widehat S (\widetilde B'\widetilde B/n)^{-} - [SS']^{-1}S \| + \|[SS']^{-1}S \|) \|\widetilde B'e/n\|
\end{eqnarray}
whence by Lemma \ref{mat perturb lem} (under condition (ii) of the Theorem), wpa1
\begin{eqnarray}
 T_1 & \lesssim & \left\{ \sigma_{JK}^{-1} \|(\widetilde B'\widetilde B/n) - I_K\| +  \sigma_{JK}^{-2}  \left(\|\widehat S - S\| +  \|(\widetilde B'\widetilde B/n) - I_K\| \right) + \sigma_{JK}^{-1} \right\} \|\widetilde B'e/n\|\,.
\end{eqnarray}
Noting that $\|\widetilde B'e/n\| = O_p(\sqrt{K/n})$ (by Markov's inequality under Assumptions \ref{resid} and \ref{b sieve}), it follows by conditions (i) and (ii) of the Theorem that $T_1 = o_p(1)$. Therefore,
for any fixed $\bar M > 0$ we have
\begin{equation}
 \limsup_{n \to \infty} \mathbb P\left(\|[\widehat S(\widetilde B'\widetilde B/n)^{-} \widehat S']^{-} \widehat S (\widetilde B'\widetilde B/n)^{-} \widetilde B'e/n\|  > \bar M\right) =0\,.
\end{equation}

Let $\mathscr B_n$ denote the event $\|[\widehat S(\widetilde B'\widetilde B/n)^{-} \widehat S']^{-} \widehat S (\widetilde B'\widetilde B/n)^{-} \widetilde B'e/n\|  \leq \bar M$ and observe that $\mathbb P(\mathscr B_n^c) = o(1)$. On $\mathscr B_n$,  for any $C \geq 1$, a finite positive $\beta = \beta(C)$ and $\gamma = \gamma(C)$ can be chosen such that
\begin{equation}
 J^\alpha \|y_0 - y_1\| \|[\widehat S(\widetilde B'\widetilde B/n)^{-} \widehat S']^{-} \widehat S (\widetilde B'\widetilde B/n)^{-} \widetilde B'e/n\| \leq C \sigma_{JK}^{-1} \sqrt{K (\log n)/n}
\end{equation}
whenever $\|y_0 - y_1\| \leq \beta n^{-\gamma}$. Let $\mathcal S_n$ be the smallest subset of $\mathcal Y_2$ such that for each $y \in \mathcal Y_2$ there exists a $y_n \in \mathcal S_n$ with $\|y_n - y\| \leq \beta n^{-\gamma}$. For any $y \in \mathcal Y_2$ let $y_n(y)$ denote the $y_n \in \mathcal S_n$ nearest (in Euclidean distance) to $y$. Therefore,
\begin{equation} \label{bn condition}
 |\widehat h(y) - P_n h_0(y) - (\widehat h(y_n(y)) - P_n h_0(y_n(y)))| \leq C \sigma_{JK}^{-1} \sqrt{K (\log n)/n}
\end{equation}
for any $y \in \mathcal Y_2$, on $\mathscr B_n$.

For any $C \geq 1$, straightforward arguments yield
\begin{eqnarray}
 & & \mathbb P \left( \|\widehat h - P_n h_0 \|_\infty \geq 4C \sigma_{JK}^{-1}\sqrt{K (\log n)/n}  \right) \notag \\
 & \leq &  \mathbb P \left( \left\{ \|\widehat h - P_n h_0 \|_\infty \geq 4C \sigma_{JK}^{-1}\sqrt{K (\log n)/n} \right\} \cap \mathscr B_n  \right) + \mathbb P(\mathscr B_n^c) \\
 & \leq &  \mathbb P \left( \left\{ \sup_{y \in \mathcal Y_2} |\widehat h(y) - P_n h_0(y) - (\widehat h(y_n(y)) - P_n h_0(y_n(y))) | \geq 2C \sigma_{JK}^{-1}\sqrt{K (\log n)/n} \right\} \cap \mathscr B_n  \right) \notag \\
 & & + \mathbb P \left( \left\{ \max_{y_n \in \mathcal S_n} |\widehat h(y_n) - P_n h_0(y_n) | \geq 2C \sigma_{JK}^{-1}\sqrt{K (\log n)/n} \right\} \cap \mathscr B_n  \right) + \mathbb P(\mathscr B_n^c) \\
 & = & \mathbb P \left( \left\{ \max_{y_n \in \mathcal S_n} |\widehat h(y_n) - P_n h_0(y_n) | \geq 2C \sigma_{JK}^{-1}\sqrt{K (\log n)/n} \right\} \cap \mathscr B_n  \right) + o(1)
\end{eqnarray}
where the final line is by (\ref{bn condition}) and the fact that $\mathbb P(\mathscr B_n^c) = o(1)$. For the remaining term:
\begin{eqnarray}
 & & \mathbb P \left( \left\{ \max_{y_n \in \mathcal S_n} |\widehat h(y_n) - P_n h_0(y_n) | \geq 2C \sigma_{JK}^{-1}\sqrt{K (\log n)/n} \right\} \cap \mathscr B_n  \right) \notag \\
 & \leq & \mathbb P \left(\max_{y_n \in \mathcal S_n} |\widetilde \psi^J(y_n)' [\widehat S(\widetilde B'\widetilde B/n)^{-} \widehat S']^{-} \widehat S (\widetilde B'\widetilde B/n)^{-} \widetilde B'e/n| \geq 2C \sigma_{JK}^{-1}\sqrt{K (\log n)/n} \right)  \\
 & \leq & \mathbb P \left(\max_{y_n \in \mathcal S_n}  |\widetilde \psi^J(y_n)' \{[\widehat S(\widetilde B'\widetilde B/n)^{-} \widehat S']^{-} \widehat S (\widetilde B'\widetilde B/n)^{-} - [SS']^{-1}S \} \widetilde B'e/n| \geq C \sigma_{JK}^{-1}\sqrt{K (\log n)/n} \right) \label{psiterm1} \\
 &  & + \mathbb P \left(\max_{y_n \in \mathcal S_n} |\widetilde \psi^J(y_n)' [S S']^{-1} S \widetilde B'e/n| \geq C \sigma_{JK}^{-1}\sqrt{K (\log n)/n} \right)  \,.\label{psiterm2}
\end{eqnarray}
It is now shown that a sufficiently large $C$ can be chosen to make terms (\ref{psiterm1}) and (\ref{psiterm2}) arbitrarily small as $n,J,K \to \infty$. Observe that $\mathcal S_n$ has cardinality $\lesssim n^\nu$ for some  $\nu = \nu(C) \in (0,\infty)$ under Assumption \ref{data}(ii).

\textbf{Control of (\ref{psiterm1}):} The Cauchy-Schwarz inequality and Assumption \ref{f sieve} yield
\begin{eqnarray}
 & & |\widetilde \psi^J(y_n)' \{[\widehat S(\widetilde B'\widetilde B/n)^{-} \widehat S']^{-} \widehat S (\widetilde B'\widetilde B/n)^{-} - [SS']^{-1}S \} \widetilde B'e/n| \notag \\
 & \lesssim & \sqrt J \|[\widehat S(\widetilde B'\widetilde B/n)^{-} \widehat S']^{-} \widehat S (\widetilde B'\widetilde B/n)^{-} - [SS']^{-1}S \| \times O_p (\sqrt{K/n})
\end{eqnarray}
uniformly for $y_n \in \mathcal S_n$ (recalling that $\|\widetilde B'e/n\| = O_p(\sqrt{K/n})$ under Assumptions \ref{resid} and \ref{b sieve}). Therefore, (\ref{psiterm1}) will vanish asymptotically provided
\begin{equation}
 T_2 := \sigma_{JK} \sqrt J \|[\widehat S(\widetilde B'\widetilde B/n)^{-} \widehat S']^{-} \widehat S (\widetilde B'\widetilde B/n)^{-} - [SS']^{-1}S \| /\sqrt{\log n} = o_p(1)\,.
\end{equation}
Under condition (ii), the bound
\begin{equation}
 T_2 \lesssim \sqrt J \left\{ \|(\widetilde B'\widetilde B/n) - I_K\| +  \sigma_{JK}^{-1}  \left(\|\widehat S - S\| +  \|(\widetilde B'\widetilde B/n) - I_K\| \right) \right\} /\sqrt{\log n}
\end{equation}
holds wpa1 by Lemma \ref{mat perturb lem}, and so $T_2 = o_p(1)$ by virtue of conditions (i) and (ii) of the Theorem.

\textbf{Control of (\ref{psiterm2}):} Let $\{M_n : n \geq 1\}$ be an increasing sequence diverging to $+\infty$ and define
\begin{equation} \begin{array}{rcl}
 \epsilon_{1,i,n} & = & \epsilon_i\{|\epsilon_i| \leq M_n\} \\
 \epsilon_{2,i,n} & = & \epsilon_i - \epsilon_{1,i,n} \\
 g_{i,n} & = & \psi^J(y_n)' \left[S S'\right]^{-1} S \widetilde b^K(X_i)\,. \end{array}
\end{equation}
Simple application of the triangle inequality yields
\begin{eqnarray}
 (\ref{psiterm2})\!\!\! & \leq & \!\!\!(\# \mathcal S_n) \max_{y_n \in \mathcal S_n} \mathbb P \left(  \left\{ \left| \sum_{i=1}^n g_{i,n} (\epsilon_{1,i,n} - E[\epsilon_{1,i,n}|\mathcal F_{i-1}]) \right | > \frac{C}{3} \sigma_{JK}^{-1} \sqrt{K (\log n)/n} \right\} \cap \mathscr A_n \right) + \mathbb P(\mathscr A_n^c)  \label{trunc npiv 1} \\
 & & + \mathbb P \left(\max_{y_n \in \mathcal S_n} \left|\frac{1}{n} \sum_{i=1}^n g_{i,n} E[\epsilon_{1,i,n}|\mathcal F_{i-1}] \right| \geq \frac{C}{3}  \sigma_{JK}^{-1}\sqrt{K (\log n)/n}\right) \label{trunc npiv 2} \\
 & & + \mathbb P \left(\max_{y_n \in \mathcal S_n} \left|\frac{1}{n} \sum_{i=1}^n g_{i,n} \epsilon_{2,i,n} \right| \geq \frac{C}{3}  \sigma_{JK}^{-1}\sqrt{K (\log n)/n}\right) \label{trunc npiv 3}
\end{eqnarray}
where $\mathscr A_n$ is a measurable set to be defined.  The following shows that terms (\ref{trunc npiv 1}), (\ref{trunc npiv 2}), and (\ref{trunc npiv 3}) vanish asymptotically provided a sequence $\{M_n : n \geq 1\}$ may be chosen such that $\sqrt{nJ/\log n} = O(M_n^{1+\delta})$ and $M_n = O(\sqrt{ n/(J \log n)})$ and $J \leq K$. Choosing $J \leq K$ and setting $M_n^{1+\delta} \asymp \sqrt{nK/\log n}$ trivially satisfies the condition $\sqrt{nK/\log n} = O(M_n^{1+\delta})$. The condition $M_n = O(\sqrt{ n/(K \log n)})$ is satisfied for this choice of $M_n$ provided $K \lesssim (n/\log n)^{\delta/(2+\delta)}$.

\textbf{Control of (\ref{trunc npiv 2}) and (\ref{trunc npiv 3}):} For term (\ref{trunc npiv 3}), first note that
\begin{equation}
 |g_{i,n}| \lesssim \sigma_{JK}^{-1} \sqrt{JK}
\end{equation}
whenever $\sigma_{JK} > 0$ by the Cauchy-Schwarz inequality, and Assumptions \ref{f sieve} and \ref{b sieve}. This, together with Markov's inequality and Assumption \ref{resid}(iii) yields
\begin{eqnarray}
 \mathbb P \left(\max_{y_n \in \mathcal S_n} \left|\frac{1}{n} \sum_{i=1}^n g_{i,n} \epsilon_{2,i,n} \right| \geq \frac{C}{3} \sigma_{JK}^{-1} \sqrt{K (\log n)/n}\right) & \lesssim & \frac{\sigma_{JK}^{-1} \sqrt{JK}E[|\epsilon_i|\{|\epsilon_i| > M_n\}]}{\sigma_{JK}^{-1} \sqrt{K (\log n)/n}}  \\
 & \leq & \sqrt{\frac{nJ}{\log n}} \frac{E[|\epsilon_i|^{2+\delta}\{|\epsilon_i| > M_n\}]}{M_n^{1+\delta}}
\end{eqnarray}
which is $o(1)$ provided $\sqrt{nJ/\log n} = O(M_n^{1+\delta})$. Term (\ref{trunc npiv 2}) is controlled by an identical argument, using the fact that $E[\epsilon_{1,i,n}|\mathcal F_{i-1}] = -E[\epsilon_{2,i,n}|\mathcal F_{i-1}]$ by Assumption \ref{resid}(i).

\textbf{Control of (\ref{trunc npiv 1}):} Term (\ref{trunc npiv 1}) is to be controlled using an exponential inequality for martingales due to \cite{vandeGeer1995}. Let $\mathscr A_n$ denote the set on which $\|(\widetilde B'\widetilde B/n) - I_K\| \leq \frac{1}{2}$ and observe that $\mathbb P(\mathscr A_n^c) = o(1)$ under the condition $\|(\widetilde B'\widetilde B/n) - I_K\| = o_p(1)$. Under Assumptions \ref{resid}(ii), \ref{f sieve}, and \ref{b sieve}, the predictable variation of the summands in (\ref{trunc npiv 1}) may be bounded by
\begin{eqnarray}
  \frac{1}{n^2}\sum_{i=1}^n E[ (g_{i,n}(\epsilon_{1,i,n} - E[\epsilon_{1,i,n}|\mathcal F_{i-1}]) )^2|\mathcal F_{i-1}]
  & \lesssim & n^{-1} \widetilde \psi^J(y_n)' [SS']^{-1} S (\widetilde B'\widetilde B/n) S'[ SS']^{-1} \widetilde \psi^J(y_n)\\
 & \lesssim & \sigma_{JK}^{-2} J/n \quad \mbox{on $\mathscr A_n$}
\end{eqnarray}
uniformly for $y_n \in \mathcal S_n$. Moreover, under Assumption \ref{b sieve}, each summand is bounded uniformly for $y_n \in \mathcal S_n$ by
\begin{equation}
 |n^{-1} g_{i,n}(\epsilon_{1,i,n} - E[\epsilon_{1,i,n}|\mathcal F_{i-1}])| \lesssim \frac{\sigma_{JK}^{-1} \sqrt{JK} M_n}{n}\,.
\end{equation}
Lemma 2.1 of \cite{vandeGeer1995} then provides that (\ref{trunc npiv 1}) may be bounded by
\begin{eqnarray}
 & & (\# \mathcal S_n) \max_{y_n \in \mathcal S_n} \mathbb P \left( \left\{ \left| \sum_{i=1}^n g_{i,n} (\epsilon_{1,i,n} - E[\epsilon_{1,i,n}|\mathcal F_{i-1}]) \right | > \frac{C}{3} \sigma_{JK}^{-1} \sqrt{K (\log n)/n} \right\} \cap \mathscr A_n \right) + \mathbb P(\mathscr A_n^c) \notag \\
 & \lesssim & n^\nu \exp \left\{- \frac{C \sigma_{JK}^{-2} K (\log n)/n}{c_1 \sigma_{JK}^{-2} J/n + c_2 n^{-1} \sigma_{JK}^{-2} \sqrt{JK} M_n \sqrt{C K(\log n)/n} } \right\} + o(1) \\
 & \lesssim & \exp \left\{\log n - \frac{C K (\log n)/n}{c_3 J/n} \right\} + \exp \left\{\log n - \frac{\sqrt{C K(\log n)/n}}{c_4K M_n/n} \right\} + o(1)
\end{eqnarray}
for finite positive constants $c_1,\ldots,c_4$. Thus (\ref{trunc npiv 1}) is $o(1)$ for large enough $C$ by virtue of the conditions $M_n = O(\sqrt{ n/(J \log n)})$ and $J \leq K$.
\end{proof}

\subsection{Proofs for Section \ref{o-npiv sec}}

\begin{proof}[Proof of Theorem \ref{sup norm rate npiv new}]
Theorem \ref{sup norm rate gen new} gives $\|\widehat h - P_n h_0\|_{\infty} = O_p (\sigma_{JK}^{-1}\sqrt{K (\log n)/n})$ provided the conditions of Theorem \ref{sup norm rate gen new} are satisfied. The conditions $J \leq K$ and $K \lesssim (n/\log n)^{\delta/(2+\delta)}$ are satisfied by hypothesis. Corollary \ref{troppcor} (under Assumptions \ref{f sieve} and \ref{b sieve} and the fact that $\{(X_i,Y_{2i})\}_{i=1}^n$ are i.i.d. and $J \leq K$) yields
\begin{eqnarray}
 \| (\widetilde B'\widetilde B/n) - I_K\| & = & O_p(\sqrt{K (\log K)/n}) \label{b rate} \\
 \| \widehat S - S\| & = & O_p(\sqrt{K (\log K)/n})\,. \label{s rate}
\end{eqnarray}
Therefore, the conditions of Theorem \ref{sup norm rate gen new} are satisfied by these rates and the conditions on $J$ and $K$ in Theorem \ref{sup norm rate npiv new}.

It remains to control the approximation error $\|P_n h_0 - h_0\|_{\infty}$. Under Assumptions \ref{data}, \ref{f sieve} (with $\Psi_J = \mbox{BSpl}(J,[0,1]^d,\gamma ) ~or~ \mbox{Wav}(J,[0,1]^d, \gamma)$) and \ref{parameter regression} there exists a $\pi_J h_0 = \widetilde \psi^{J \prime} c_J \in \Psi_J$ with $c_J \in \mathbb R^J$ such that
\begin{equation} \label{infty approximation}
 \|h_0 - \pi_J h_0\|_{\infty} = O(J^{-p/d})
\end{equation}
(see, e.g., \cite{Huang1998}) so it suffices to control $\|P_n h_0 - \pi_J h_0\|_{\infty}$.

Both $P_n h_0$ and $\pi_J h_0$ lie in $\Psi_J$, so $\|P_n h_0 - \pi_J h_0\|_{\infty}$ may be rewritten as
\begin{eqnarray}
 \|P_n h_0 - \pi_J h_0\|_{\infty}
 & = & \frac{\|P_n h_0 - \pi_J h_0\|_{\infty}}{\| \Pi_K T(P_n h_0-\pi_J h_0)\|_{\infty}}\times \| \Pi_K T(P_n h_0-\pi_J h_0)\|_{\infty}  \\
 & \leq & \sigma_{\infty,JK}^{-1} \times \| \Pi_K T(P_n h_0-\pi_J h_0)\|_{\infty} \label{delta p target}
\end{eqnarray}
where
\begin{equation}
 \Pi_K T(P_n h_0-\pi_J h_0)(x) = \widetilde b^K(x) S' [\widehat S (\widetilde B'\widetilde B/n)^{-} \widehat S]^{-} \widehat S( \widetilde B'\widetilde B/n)^{-}\widetilde B' (H_0 - \Psi c_J)/n\,.
\end{equation}
Define the $K \times K$ matrices
\begin{equation} \begin{array}{rcl}
 D & = & S'[SS']^{-1}S\\
 \widehat D & = & (\widetilde B'\widetilde B/n)^{-} \widehat S '[\widehat S (\widetilde B'\widetilde B/n)^{-} \widehat S]^{-} \widehat S( \widetilde B'\widetilde B/n)^{-}\,. \label{D def}
\end{array}
\end{equation}
By the triangle inequality,
\begin{eqnarray}
 & & \| \Pi_K T(P_n h_0-\pi_J h_0)\|_{\infty} \\
 & \leq & \|\widetilde b^K(x) \widehat D \widetilde B'(H_0 - \Psi c_J)/n\|_\infty \label{bias term 1}\\
 & & + \|\widetilde b^K(x) \{S' -  (\widetilde B'\widetilde B/n)^{-} \widehat S ' \} [\widehat S (\widetilde B'\widetilde B/n)^{-} \widehat S]^{-} \widehat S( \widetilde B'\widetilde B/n)^{-} \widetilde B'(H_0 - \Psi c_J)/n\|_\infty \,.\label{bias term 2}
\end{eqnarray}
The arguments below will show that
\begin{equation} \label{pi_K bound}
 \| \Pi_K T(P_n h_0-\pi_J h_0)\|_{\infty} = O_p( \sqrt{K(\log n)/n}) \times \|h_0 - \pi_J h_0\|_\infty  + O_p(1) \times \| \Pi_K T(h_0 - \pi_J h_0)\|_\infty\,.
\end{equation}
Substituting (\ref{pi_K bound}) into (\ref{delta p target}) and using Assumption \ref{modulus}(ii), the bound $\sigma_{\infty,JK}^{-1} \lesssim \sqrt J \times \sigma_{JK}^{-1}$ (under Assumption \ref{f sieve}), equation (\ref{infty approximation}), and the condition $p \geq d/2$ in Assumption \ref{parameter regression} yields the desired result
\begin{eqnarray}
 \|P_n h_0 - \pi_J h_0\|_{\infty} 
 & = & O_p(J^{-p/d} + \sigma_{JK}^{-1} \sqrt{K(\log n)/n})\,.
\end{eqnarray}

\textbf{Control of (\ref{bias term 1}):} By the triangle and Cauchy-Schwarz inequalities and compatibility of the spectral norm under multiplication,
\begin{eqnarray}
 (\ref{bias term 1}) & \leq & \|\widetilde b^K(x) (\widehat D - D) \widetilde B'(H_0 - \Psi c_J)/n\|_\infty \notag \\
 & & + \|\widetilde b^K(x) D \{ \widetilde B'(H_0 - \Psi c_J)/n - E[\widetilde b^K(X_i)(h_0(Y_{2i}) - \pi_J h_0(Y_{2i}))]\} \|_\infty \notag \\
 & & + \|\widetilde b^K(x) D E[\widetilde b^K(X_i)\Pi_K T(h_0 - \pi_J h_0)(X_i)] \|_\infty \\
 & \lesssim & \sqrt K \|\widehat D - D\| \{ O_p( \sqrt{K/n}) \times \|h_0 - \pi_J h_0\|_\infty  +  \| \Pi_K T(h_0 - \pi_J h_0)\|_{L^2(X)} \} \notag \\
 & & + O_p( \sqrt{K (\log n)/n}) \times \|h_0 - \pi_J h_0\|_\infty \notag \\
 & & +  \|\widetilde b^K(x) D E[\widetilde b^K(X_i)\Pi_K T(h_0 - \pi_J h_0)(X_i)] \|_\infty
\end{eqnarray}
by Lemma \ref{BH lemma} and properties of the spectral norm. Lemma \ref{mat perturb lem} (the conditions of Lemma \ref{mat perturb lem} are satisfied by virtue of (\ref{b rate}) and (\ref{s rate}) and the condition $\sigma_{JK}^{-1} K \sqrt{(\log n)/n} \lesssim 1$) and the condition $\sigma_{JK}^{-1} K \sqrt{(\log n)/n} \lesssim 1$ yield $\sqrt K \|\widehat D - D\| = O_p(1)$. Finally, Lemma \ref{tilde P bounded} (under Assumptions \ref{data} and \ref{b sieve}) provides that
\begin{equation}
 \|\widetilde b^K(x) D E[\widetilde b^K(X_i)\Pi_K T(h_0 - \pi_J h_0)(X_i)] \|_\infty \lesssim \| \Pi_K T (h_0 - \pi_J h_0)\|_\infty
\end{equation}
and so
\begin{equation}
 (\ref{bias term 1}) = O_p( \sqrt{K(\log n)/n}) \times \|h_0 - \pi_J h_0\|_\infty  + O_p(1) \times \| \Pi_K T (h_0 - \pi_J h_0)\|_\infty
\end{equation}
as required.

\textbf{Control of (\ref{bias term 2}):} By the Cauchy-Schwarz inequality, compatibility of the spectral norm under multiplication, and Assumption \ref{b sieve},
\begin{eqnarray}
 (\ref{bias term 2}) & \lesssim & \sqrt K \|S' -  (\widetilde B'\widetilde B/n)^{-} \widehat S '\| \|[\widehat S (\widetilde B'\widetilde B/n)^{-} \widehat S]^{-} \widehat S( \widetilde B'\widetilde B/n)^{-} \| \|\widetilde B'(H_0 - \Psi c_J)/n\| \\
 & \lesssim & \sigma_{JK}^{-1} \sqrt K \{\|(\widetilde B'\widetilde B/n) - I_K\| + \|\widehat S - S\| \} \|\widetilde B'(H_0 - \Psi c_J)/n\|
\end{eqnarray}
where the second line holds wpa1, by Lemma \ref{mat perturb lem} (the conditions of Lemma \ref{mat perturb lem} are satisfied by virtue of (\ref{b rate}) and (\ref{s rate}) and the condition $\sigma_{JK}^{-1} K \sqrt{(\log n)/n} \lesssim 1$). Applying (\ref{b rate}) and (\ref{s rate}) and the condition $\sigma_{JK}^{-1} K \sqrt{(\log n)/n} \lesssim 1$ again yields
\begin{eqnarray}
 (\ref{bias term 2}) & = & O_p(1) \times \|\widetilde B'(H_0 - \Psi c_J)/n\| \\
 & = & O_p( \sqrt{K/n}) \times \|h_0 - \pi_J h_0\|_\infty  + O_p(1) \times \| \Pi_K T (h_0 - \pi_J h_0)\|_\infty
\end{eqnarray}
by Lemma \ref{BH lemma}, equation (\ref{infty approximation}), and the relation between $\|\cdot\|_\infty$ and $\|\cdot\|_{L^2(X)}$.
\end{proof}

\begin{proof}[Proof of Lemma \ref{ill-suff}]
As already mentioned, Assumption \ref{modulus}(i) implies
that the operators $T$, $T^{\ast }$, $TT^{\ast }$ and $T^{\ast }T$ are all
compact with the singular value system $\left\{ \mu _{k};\phi _{1k},\phi
_{0k}\right\} _{k=1}^{\infty }$ where $\mu _{1}=1\geq \mu _{2}\geq \mu
_{3}\geq ...\searrow 0$. For any $h\in B(p,L)\subset L^{2}(Y_{2})$, $g\in
L^{2}(X)$, we have
\begin{equation*}
(Th)(x)=\sum_{k=1}^{\infty }\mu _{k}\left\langle h,\phi _{1k}\right\rangle
_{Y_{2}}\phi _{0k}(x)\text{,\quad }(T^{\ast }g)(y_{2})=\sum_{k=1}^{\infty
}\mu _{k}\left\langle g,\phi _{0k}\right\rangle _{X}\phi _{1k}(y_{2})\text{.}
\end{equation*}
Let $\mathcal{P}_{J}=clsp\{\phi _{0k}:k=1,...,J\}$, and note that $\mathcal P_J$ is a closed linear subspace of $B_K$ under the conditions of part (3).

To prove part (3),
for any $h\in \Psi _{J}$ with $J\leq K$, we have
\begin{eqnarray}
\Pi_K T h(\cdot) & = & \sum_{j=1}^{K}\langle Th,\widetilde{b}%
_{Kj}\rangle _{X}\widetilde{b}_{Kj}(\cdot ) \\
& = & \sum_{j=1}^{J}\langle
Th,\phi _{0j}\rangle _{X}\phi _{0j}(\cdot)+R(\cdot ,h)
\end{eqnarray}%
for some $R(\cdot ,h)\in B_{K}\backslash \mathcal{P}_{J}$.
Therefore%
\begin{eqnarray}
\sigma _{JK}^{2} &=& \inf_{h \in \Psi_J : \|h\|_{L^2(Y_2)} = 1} \|\Pi_K T h(X)\|^2_{L^2(X)} \\
& = & \inf_{h\in \Psi
_{J}:||h||_{L^2(Y_{2})}=1} \left\|\sum_{j=1}^{J}\left\langle Th,\phi _{0j}\right\rangle
_{X}\phi _{0j}(\cdot ) + R(\cdot ,h)\right\|_{L^2(X)}^{2} \\
&=&\inf_{h\in \Psi _{J}:||h||_{L^2(Y_{2})}=1}\left( \sum_{j=1}^{J}\left\langle
Th,\phi _{0j}\right\rangle _{X}^{2}+||R(\cdot ,h)||_{L^2(X)}^{2}\right)  \\
 &\geq & \inf_{h\in \Psi _{J}:||h||_{L^2(Y_{2})}=1}\left(
\sum_{j=1}^{J}\left\langle Th,\phi _{0j}\right\rangle _{X}^{2}\right) \\
&=&\inf_{h\in \Psi _{J}:||h||_{L^2(Y_{2})}=1}\left( \sum_{j=1}^{J}\mu
_{j}^{2}\left\langle h,\phi _{1j}\right\rangle _{Y_{2}}^{2}\right) \\
& \geq  & \mu
_{J}^{2}\inf_{h\in \Psi _{J}:||h||_{L^2(Y_{2})}=1}\left(
\sum_{j=1}^{J}\left\langle h,\phi _{1j}\right\rangle _{Y_{2}}^{2}\right)
\quad = \quad  \mu _{J}^{2}\,.
\end{eqnarray}%
This, together with part (1), gives $1/\mu_J \geq \sigma_{JK}^{-1} \geq \tau_{2,2,J} \geq 1/\mu_J$.
\end{proof}

\begin{proof}[Proof of Theorem \ref{npir lower bound}]
Our proof proceeds by application of Theorem 2.5 of \cite{Tsybakov2009} (page 99).

We first explain the scalar $(d = 1)$ case in detail. Let $\{\phi_{jk},\psi_{jk}\}_{j,k}$ be a wavelet basis for $L^2([0,1])$ as in the construction of \cite{CDJV1993,CDV1993} with regularity $\gamma > \max\{p,\varsigma\}$ using a pair $(\phi,\psi)$ for which $\mbox{support}(\phi) = \mbox{support}(\psi) = [0,2N-1]$. The precise type of wavelet is not important, but we do require that $\|\psi\|_\infty < \infty$. For given $j$, the wavelet space $W_j$ consists of $2^j$ functions $\{\psi_{jk}\}_{1 \leq k \leq 2^j}$, such that $\{\psi_{jk}\}_{1 \leq k \leq 2^j-2N}$ are interior wavelets for which $\psi_{jk}(\cdot) = 2^{j/2}\psi(2^j(\cdot)-k)$. We will choose $j$ deterministically with $n$, such that
\begin{equation}
 2^j \asymp (n/\log n)^{1/(2(p+\varsigma)+1)}\,.
\end{equation}
By construction, the support of each interior wavelet is an interval of length $2^{-j}(2N-1)$. Thus for all $j$ sufficiently large,\footnote{Hence the $\liminf$ in our statement of the Lemma.} we may choose a set $M \subset \{1,\ldots,2^j-2N\}$ of interior wavelets with $\#(M) \gtrsim 2^j$ such that $\mbox{support}(\psi_{jm}) \cap \mbox{support}(\psi_{jm'}) = \emptyset$ for all $m,m' \in M$ with $m \neq m'$. Note also that by construction we have $\#(M) \leq 2^j$ (since there are $2^j-2N$ interior wavelets).

We begin by defining a family of submodels. Let $h_0 \in B(p,L)$ be such that $\|h_0\|_{B^p_{\infty,\infty}(\mathcal Y_2)} \leq L/2$, and for each $m \in M$ let
\begin{equation}
 h_m = h_0 + c_0 2^{-j(p+1/2)} \psi_{jm}
\end{equation}
where $c_0$ is a positive constant to be defined subsequently. Noting that
\begin{eqnarray}
  c_0 2^{-j(p+1/2)} \|\psi_{jm}\|_{B^p_{\infty,\infty}} & \lesssim &  c_0 2^{-j(p+1/2)} \|\psi_{jm}\|_{b^p_{\infty,\infty}} \\
  & \leq & c_0
\end{eqnarray}
it follows by the triangle inequality that $\|h_m\|_{B^p_{\infty,\infty}} \leq L$ uniformly in $m$ for all sufficiently small $c_0$. For $m \in \{0\} \cup M$ let $P_m$ be the joint distribution of $\{(X_i,Y_{1i})\}_{i=1}^n$ with $Y_{1i} = T h_m(X_i) + u_i$ for the Gaussian NPIR model (\ref{npir}).

To apply Theorem 2.5 of \cite{Tsybakov2009}, first note that for any $m \in M$
\begin{eqnarray}
 \|h_0 - h_m\|_\infty & = & c_0 2^{-j(p+1/2)}\| \psi_{jm} \|_\infty \\
 & = & c_0 2^{-jp}\|\psi\|_\infty
\end{eqnarray}
and for any $m,m' \in M$ with $m \neq m'$
\begin{eqnarray}
 \|h_m - h_{m'}\|_\infty & = & c_0 2^{-j(p+1/2)}\| \psi_{jm} - \psi_{jm'} \|_\infty \\
 & = & 2 c_0 2^{-jp}\|\psi\|_\infty
\end{eqnarray}
by virtue of the disjoint support of $\{\psi_{jm}\}_{m \in M}$. Using the KL divergence for the multivariate normal distribution (under the Gaussian NPIR model (\ref{npir})),  Assumption \ref{smoothness} and the equivalence between the Besov function-space and sequence-space norms, the KL distance $K(P_m,P_0)$ is
\begin{eqnarray}
 K(P_m,P_0) & \leq & \frac{1}{2}\sum_{i=1}^n (c_0 2^{-j(p+1/2)})^2 E\left[\frac{(T \psi_{jm}(X_i))^2 }{\sigma^2(X_i)}\right] \\
 & \leq & \frac{1}{2}\sum_{i=1}^n (c_0 2^{-j(p+1/2)})^2 \frac{E\left[(T \psi_{jm}(X_i))^2\right] }{\sigma^2_0}  \\
 & = &  \frac{n}{2\sigma_0^2}(c_0 2^{-j(p+1/2)})^2 \|(T^*T)^{1/2}\psi_{jm}(Y_2)\|_{L^2(Y_2)}^2 \\
 & \lesssim &   \frac{n}{2\sigma_0^2}(c_0 2^{-j(p+1/2)})^2 (2^{-j\varsigma})^2 \\
 & = & \frac{n}{2\sigma_0^2} c_0^2 2^{-j(2(p+\varsigma)+1)} \\
 & \lesssim & c_0^2 \log n
\end{eqnarray}
since $2^{-j} \asymp ((\log n)/n)^{1/(2(p+\varsigma)+1)}$. Moreover, since $\#(M) \asymp 2^j$, we have
\begin{equation}
 \log(\#(M)) \lesssim \log n + \log \log n
\end{equation}
Therefore, we may choose $c_0$ sufficiently small that $\|h_m\|_{B^p_{\infty,\infty}} \leq L$ and $K(P_m,P_0) \leq \frac{1}{8} \log (\#(M))$ uniformly in $m$ for all $n$ sufficiently large. All conditions of Theorem 2.5 of \cite{Tsybakov2009} are satisfied and hence we obtain the lower bound result.

The multivarite case uses similar arguments for a tensor-product wavelet basis (see \cite{Triebel2006,Triebel2008}). We choose the same $j$ for each univariate spaces such that $2^j \asymp (n/\log n)^{1/(2(p+\varsigma)+d)}$ and so the tensor-product wavelet space has dimension $2^{jd} \asymp (n/\log n)^{d/(2(p+\varsigma)+d)}$. We construct the same family of submodels, setting $h_m = h_0 + c_0 2^{-j(p+d/2)} \psi_{jm}$ where $\psi_{jm}$ is now the product of $d$ interior univariate wavelets defined previously. Since we take the product of $d$ univariate interior wavelets, we again obtain
\begin{equation}
 \|h_m - h_{m'}\|_\infty \gtrsim c_0 2^{-jp}
\end{equation}
for each $m,m' \in \{0\} \cup M$ with $m \neq m'$, and
\begin{eqnarray}
 K(P_m,P_0) & \lesssim &  \frac{n}{2\sigma_0^2}(c_0 2^{-j(p+d/2)})^2 (2^{-j\varsigma})^2 \\
 & = & \frac{n}{2\sigma_0^2} c_0^2 2^{-j(2(p+\varsigma)+d)} \\
 & \lesssim & c_0^2 \log n\,.
\end{eqnarray}
The result follows as in the univariate case.
\end{proof}

\subsection{Proofs for Section \ref{reg sec}}

\begin{proof}[Proof of Theorem \ref{sup norm rate regression}]
The variance term is immediate from Theorem \ref{sup norm rate gen new} with $\sigma_{JK} =1$.
The bias calculation follows from \cite{Huang2003} under Assumptions \ref{data}(ii), \ref{b sieve} (with $B_K = \mbox{BSpl}(K,[0,1]^{d},{\gamma} )~or~ \mbox{Wav}(K,[0,1]^{d}, {\gamma})$), and \ref{parameter regression} and the fact that the empirical and true $L^2(X)$ norms are equivalent over $B_K$ wpa1 by virtue of the condition $\|\widetilde B'\widetilde B/n - I_K\| = o_p(1)$, which is implied by Condition (ii).
\end{proof}

\begin{proof}[Proof of Corollary \ref{regcor}]
By Theorem \ref{sup norm rate regression}, the optimal sup-norm convergence rate $(n/\log n)^{-p/(2p +d)}$ is achieved by setting $K\asymp (n/\log n)^{d/(2p +d)})$, with $\delta \geq d/p$ for condition (i) to hold.
(1) When the regressors are i.i.d., by Lemma \ref{Bconvi.i.d.}, condition (ii) is satisfied provided that $d \leq 2p$ (which is assumed in Assumption \ref{parameter regression}).
(2) When the regressors are exponentially $\beta$-mixing, by Lemma \ref{Bconvbeta}, condition (ii) is satisfied provided that $d < 2p$.
(3) When the regressors are algebraically $\beta$-mixing at rate $\gamma$, by Lemma \ref{Bconvbeta}, condition (ii) is satisfied provided that  $(2+\gamma)d < 2 \gamma p$.
\end{proof}

\subsection{Proofs for Section \ref{ei sec}}

\begin{proof}[Proof of Corollary \ref{troppcor}]
Follows by application of Theorem \ref{troppthm} with $t = C \sigma_n \sqrt{\log(d_1 + d_2)}$ for sufficiently large $C$, and applying the condition $R_n  \sqrt{\log (d_1+d_2)} = o(\sigma_n)$.
\end{proof}

\begin{proof}[Proof of Theorem \ref{beta tropp}]
 By Berbee's lemma (enlarging the probability space as necessary) the processs $\{X_i\}$ can be coupled with a process $X^*_i$ such that $Y_{k}: = \{X_{(k-1)q+1},\ldots,X_{kq}\}$ and $ Y_k^* := \{X_{(k-1)q+1}^*,\ldots,X_{kq}^*\}$ are identically distributed for each $k \geq 1$, $\mathbb P(Y_k \neq Y_k^*) \leq \beta(q)$ for each $k \geq 1$ and $\{Y_1^*,Y_3^*,\ldots\}$ are independent and $\{Y_2^*,Y_4^*,\ldots\}$ are independent (see, e.g., \cite{DoukhanMassartRio}). Let $I_e $ and $I_o$ denote the indices of $\{1,\ldots,n\}$ corresponding to the odd- and even-numbered blocks, and $I_r$ the indices in the remainder, so $I_r =
q[n/q]+1,\ldots,n$ when $q[n/q] < n$ and $I_r = \emptyset$ when $q[n/q] = n$.

Let $\Xi^*_{i,n} = \Xi(X^*_{i,n})$. By the triangle inequality,
\begin{equation} \begin{array}{rcl}
 & & \mathbb P \left( \|\sum_{i=1}^n \Xi_{i,n} \| \geq 6t \right) \\
 & \leq & \mathbb P ( \|\sum_{i=1}^{[n/q]q} \Xi_{i,n}^*\| + \|\sum_{i \in I_r} \Xi_{i,n}\| + \|\sum_{i=1}^{[n/q]q} (\Xi_{i,n}^* - \Xi_{i,n}) \| \geq 6t ) \\
 & \leq & \frac{n}{q} \beta (q) + \mathbb P \left( \|\sum_{i \in I_r} \Xi_{i,n} \| \geq t \right) + \mathbb P \left( \|\sum_{i \in I_e} \Xi_{i,n}^* \| \geq t \right) + \mathbb P \left( \|\sum_{i \in I_o} \Xi_{i,n}^* \| \geq t \right)
\end{array}
\end{equation}
To control the last two terms we apply Theorem \ref{troppthm}, recognizing that $\sum_{i \in I_e} \Xi_{i,n}^*$ and $\sum_{i \in I_o} \Xi_{i,n}^*$ are each the sum of fewer than $[n/q]$ independent $d_1 \times d_2$ matrices, namely $W_k^*= \sum_{i=(k-1)q+1}^{kq} \Xi_{i,n}^*$. Moreover each $W_k^*$ satisfies $\|W_k^*\| \leq q R_n$ and $\max\{\|E[W_k^* W_k^{*\prime}]\|,\|E[W_k^{*\prime}W_k^*]\|\} \leq q^2 s_n$. Theorem \ref{troppthm} then yields
\begin{equation}
 \mathbb P \left( \left\|\sum_{i \in I_e} \Xi_{i,n}^* \right\| \geq t \right) \leq (d_1 + d_2) \exp \left( \frac{-t^2/2}{nq s_n^2 + qR_n t/3} \right)
\end{equation}
and similarly for $I_o$.
\end{proof}

\begin{proof}[Proof of Corollary \ref{beta rate}]
Follows by application of Theorem \ref{beta tropp} with $t = C s_n \sqrt{n q \log(d_1 + d_2)}$ for sufficiently large $C$, and applying the conditions $\frac{n}{q}\beta(q) = o(1)$ and $R_n \sqrt{q \log(d_1+d_2)} = o(s_n\sqrt{n})$.
\end{proof}

\begin{proof}[Proof of Lemma \protect\ref{eilem}]
Let $G_K = E[b^{K}(X)b^{K}(X)^{\prime
}]$. Since $B_{K}=clsp \{b_{1},\ldots ,b_{K}\}$ , we have:%
\begin{eqnarray}
&&\sup \{\textstyle|\frac{1}{n}\sum_{i=1}^{n}b(X_{i})^{2}-1|:b\in
B_{K},E[b(X)^{2}]=1\} \notag  \\
&=&\sup \{|c^{\prime }(B^{\prime }B/n-G_{K})c|:c\in \mathbb{R}^{K},\Vert
G_{K}^{1/2}c\Vert =1\} \\
&=&\sup \{|c^{\prime }G_{K}^{1/2}(G_{K}^{-1/2}(B^{\prime
}B/n)G_{K}^{-1/2}-I_{K})G_{K}^{1/2}c|:c\in \mathbb{R}^{K},\Vert
G_{K}^{1/2}c\Vert =1\} \\
&=&\sup \{|c^{\prime }(\widetilde{B}^{\prime }\widetilde{B}/n-I_{K})c|:c\in
\mathbb{R}^{K},\Vert c\Vert =1\} \\
&=&\Vert \widetilde{B}^{\prime }\widetilde{B}/n-I_{K}\Vert _{2}^{2}
\end{eqnarray}%
as required.
\end{proof}

\begin{proof}[Proof of Lemma \protect\ref{Bconvi.i.d.}]
Follows by application of Corollary \ref{troppcor} with $\Xi_{i,n} = n^{-1}(\widetilde b^K(X_i) \widetilde b^K(X_i)^{\prime }- I_K)$, $R_n \lesssim n^{-1} (\zeta_0(K)^2+1)$, and $\sigma_n^2 \lesssim n^{-1}(\zeta_0(K)^2+1)$.
\end{proof}

\begin{proof}[Proof of Lemma \protect\ref{Bconvbeta}]
Follows by application of Corollary \ref{beta rate} with $\Xi_{i,n} = n^{-1}(\widetilde b^K(X_i) \widetilde b^K(X_i)^{\prime }- I_K)$, $R_n \lesssim n^{-1} (\zeta_0(K)^2+1)$, and $s_n^2 \lesssim n^{-2}(\zeta_0(K)^2+1)$.
\end{proof}

\section{Supplementary lemmas and their proofs}

\cite{Huang2003} provides conditions under which the operator norm of orthogonal projections onto sieve spaces are \emph{stable} in sup norm as the dimension of the sieve space increases. The following Lemma shows the same is true for the operator $Q_K : L^\infty(X) \to L^\infty(X)$ given by
\begin{equation}
 Q_K u(x) = \widetilde b^K(x)D E[\widetilde b^K(X) u(X)]
\end{equation}
where $D = S'[SS']^{-1}S$, i.e.
\begin{equation}
 \limsup_{K \to \infty} \sup_{u \in L^\infty(X)} \frac{\|Q_K u\|_\infty}{\|u\|_\infty} \leq C
\end{equation}
for some finite positive constant $C$. The proof follows by simple modification of the arguments in Theorem A.1 in \cite{Huang2003} (see also Corollary A.1 of \cite{Huang2003}).

\begin{lemma}\label{tilde P bounded}
$Q_K$ is stable in sup norm under Assumption \ref{data} and \ref{b sieve}.
\end{lemma}

\begin{proof}[Proof of Lemma \ref{tilde P bounded}]
The assumptions of Theorem A.1 of \cite{Huang2003} with $\nu$ and $\nu_n$ taken to be the distribution of $X$ are satisfied under Assumption \ref{data}. Let $P_K$ denote the orthogonal projection onto the sieve space, i.e.
\begin{equation}
 P_K u(x) = b^K(x)'E[b^K(X) b^K(X)']^{-1} E[b^K(X)u(X)]
\end{equation}
for any $u \in L^\infty(X)$. Let $\langle \cdot,\cdot \rangle$ denote the $L^2(X)$ inner product. Since $D$ is an orthogonal projection matrix and $P_K$ is a $L^2(X)$ orthogonal projection onto $B_K$, for any $u \in L^\infty(X)$
\begin{equation} \begin{array}{rcl}
  \| Q_K u\|^2_{L^2(X)}
 & = & E[u(X) \widetilde b^K(X)']D^2 E[\widetilde b^K(X)u(X)] \\
 & \leq & E[u(X) \widetilde b^K(X)'] E[\widetilde b^K(X)u(X)]\\
 & = & \|P_K u\|^2_{L^2(X)} \\
 & \leq & \|u\|^2_{L^2(X)}\,. \end{array} \label{contract}
\end{equation}

As in \cite{Huang2003}, let $\Delta$ index a partition of $\mathcal X$ into finitely many polyhedra. Let $v \in L^\infty(X)$ be supported on $\delta_0$ for some $\delta_0 \in \Delta$ (i.e. $v(x) = 0$ if $x \not \in \delta_0$). For some coefficients $\alpha_1,\ldots,\alpha_K$,
\begin{equation}
 Q_K v (x) = \sum_{i=1}^K \alpha_i b_{Ki} (x)\,.
\end{equation}
Let $d(\cdot,\cdot)$ be the distance measure between elements of $\Delta$ defined in the Appendix of \cite{Huang2003}. Let $l$ be a nonnegative integer and let $I_l \subset \{1,\ldots,K\}$ be the set of indices such that for any $i \in I_l$ the basis function $b_{Ki}$ is active on a $\delta \in \Delta$ with $d(\delta,\delta_0) \leq l$. Finally, let
\begin{equation}
 v_l(x) = \sum_{i \in I_l} \alpha_i b_{Ki}(x)\,.
\end{equation}

For any $v \in L^\infty(X)$,
\begin{eqnarray}
 & & \|Q_K v\|_{L^2(X)}^2 \notag \\
 & = & \|Q_K v - v\|_{L^2(X)}^2 +\|v\|_{L^2(X)}^2 + 2 \langle Q_K v - v, v\rangle \\
 & = & \|P_K v - v\|_{L^2(X)}^2 + \|Q_K v - P_K v\|^2_{L^2(X)}   + 2 \langle Q_K v - P_K v, P_K v - v \rangle +\|v\|_{L^2(X)}^2 + 2 \langle Q_K v - v, v\rangle \\
 & \leq & \|v_l - v\|_{L^2(X)}^2+ \|Q_K v - P_K v\|^2_{L^2(X)}   + 2 \langle Q_K v - P_K v, P_K v - v \rangle +\|v\|_{L^2(X)}^2 + 2 \langle Q_K v - v, v\rangle \label{b2 ineq 1} \\
 & = & \|v_l - v\|_{L^2(X)}^2+  \|v\|_{L^2(X)}^2 + \langle Q_K v - P_K v, Q_K v + P_K v - 2v \rangle + 2\langle Q_K v - v,v \rangle \\
 & = & \|v_l - v\|_{L^2(X)}^2+  \|v\|_{L^2(X)}^2 + \| Q_K v \|_{L^2(X)}^2 - \|P_K v\|^2_{L^2(X)} - 2 \langle Q_K v - P_K v,v\rangle + 2\langle Q_K v - v,v \rangle \\
 & = & \|v_l - v\|_{L^2(X)}^2+  \|v\|_{L^2(X)}^2 + \| Q_K v \|_{L^2(X)}^2 - \|P_K v\|^2_{L^2(X)} + 2 \langle P_K v - v,v\rangle  \\
 & = & \|v_l - v\|_{L^2(X)}^2+  \|v\|_{L^2(X)}^2 + \| Q_K v \|_{L^2(X)}^2 + \|P_K v\|^2_{L^2(X)} -2\| v\|_{L^2(X)}  \\
 & \leq & \|v_l - v\|_{L^2(X)}^2 + \|v\|_{L^2(X)}^2 \label{b2 ineq 2}
\end{eqnarray}
where (\ref{b2 ineq 1}) uses the fact that $P_K$ is an orthogonal projection, and (\ref{b2 ineq 2}) follows from  (\ref{contract}). The remainder of the proof of Theorem A.1 of \cite{Huang2003} goes through under these modifications.
\end{proof}

The next lemma provides useful bounds on the estimated matrices encountered in the body of the paper. Recall the definitions of $\widehat D$ and $D$ in expression (\ref{D def}).

\begin{lemma}\label{mat perturb lem}
Under Assumption \ref{f sieve}(ii) and \ref{b sieve}(ii), if $J \leq K$,  $\|(\widetilde B'\widetilde B/n) - I_K\| = o_p(1)$, then wpa1
\begin{enumerate}[(i)]
\item $(\widetilde B'\widetilde B/n)$ is invertible and $\|(\widetilde B'\widetilde B/n)^{-1}\| \leq 2$
\item $\|(\widetilde B'\widetilde B/n)^{-} \widehat S' - S'\| \lesssim \|(\widetilde B'\widetilde B/n) - I_K\| + \|\widehat S - S\| $
\item $\|(\widetilde B'\widetilde B/n)^{-1/2} \widehat S' - S'\| \lesssim \|(\widetilde B'\widetilde B/n) - I_K\| + \|\widehat S - S\|$.
\end{enumerate}
If, in addition, $\sigma_{JK}^{-1}(\|(\widetilde B'\widetilde B/n) - I_K\| + \|\widehat S - S\|) = o_p(1)$, then wpa1
\begin{enumerate}[(i)] \setcounter{enumi}{3}
\item $(\widetilde B'\widetilde B/n)^{-1/2} \widehat S'$ has full column rank and $\widehat S (\widetilde B'\widetilde B/n)^{-1} \widehat S$ is invertible
\item $\|\widehat D - D\|  \lesssim \|(\widetilde B'\widetilde B/n) - I_K\|  + \sigma_{JK}^{-1} (\|\widehat S - S\| +  \|(\widetilde B'\widetilde B/n) - I_K\| )$
\item $\| [\widehat S(\widetilde B'\widetilde B/n)^{-} \widehat S']^{-} \widehat S (\widetilde B'\widetilde B/n)^{-} - [SS']^{-1}S \| \lesssim  \sigma_{JK}^{-1} \|(\widetilde B'\widetilde B/n) - I_K\| +  \sigma_{JK}^{-2}  (\|\widehat S - S\| +  \|(\widetilde B'\widetilde B/n) - I_K\| )$.
\end{enumerate}
\end{lemma}

\begin{proof}[Proof of Lemma \ref{mat perturb lem}]
We prove Lemma \ref{mat perturb lem} by part. Note that under Assumption \ref{f sieve}(ii) and \ref{b sieve}(ii), $\|S\| \leq 1$ since $S$ is isomorphic to the $L^2(X)$ orthogonal projection of $T$ onto the space $B_K$, restricted to $\Psi_J$.
\begin{enumerate}[(i)]
\item  Let $\mathscr A_n$ denote the event $\|(\widetilde B'\widetilde B/n) - I_K\| \leq \frac{1}{2}$. The condition $\|(\widetilde B'\widetilde B/n) - I_K\| = o_p(1)$ implies that $\mathbb P(\mathscr A_n^c) = o(1)$. Clearly $\|(\widetilde B'\widetilde B/n)^{-1}\| \leq 2 $ on $\mathscr A_n$.

\item Working on $\mathscr A_n$ (so we replace the generalized inverse with an inverse), Assumption \ref{b sieve}(ii), the triangle inequality, and compatibility of the spectral norm under multiplication yields
\begin{eqnarray}
 \|(\widetilde B'\widetilde B/n)^{-1} \widehat S' - S'\| & \leq & \|(\widetilde B'\widetilde B/n)^{-1} \widehat S' - (\widetilde B'\widetilde B/n)^{-1} S'\| + \|(\widetilde B'\widetilde B/n)^{-1} S' - S'\| \\
 & \leq & \|(\widetilde B'\widetilde B/n)^{-1}\|\| \widehat S - S\| + \|(\widetilde B'\widetilde B/n)^{-1}  - I_K \| \|S'\| \\
  & \leq & 2 \| \widehat S - S\| + \|(\widetilde B'\widetilde B/n)^{-1}  - I_K \| \\
  & = & 2\| \widehat S - S\| + \|(\widetilde B'\widetilde B/n)^{-1}[(\widetilde B'\widetilde B/n) - I_K]\| \\
  & \leq & 2\| \widehat S - S\| + 2\|(\widetilde B'\widetilde B/n) - I_K\| 
\end{eqnarray}
\item Follows the same arguments as (ii), noting additionally that $\lambda_{\min} ((\widetilde B'\widetilde B/n)^{-1}) \leq \frac{1}{2}$ on $\mathscr A_n$, in which case
\begin{equation} \label{schmitt bound}
 \|(\widetilde B'\widetilde B/n)^{-1/2} - I_K\| \leq (1 + 2^{-1/2})^{-1}\| (\widetilde B'\widetilde B/n) - I_K\|
\end{equation}
by Lemma 2.2 of \cite{Schmitt}.

\item Let $s_{J}(A)$ denote the $J$th largest singular value of a $J \times K$ matrix $A$. Weyl's inequality yields
\begin{equation}
 |s_J(\widehat S (\widetilde B'\widetilde B/n)^{-1/2}) - \sigma_{JK}| \leq \|(\widetilde B'\widetilde B/n)^{-1/2} \widehat S' - S'\|\,.
\end{equation}
This and the condition $\sigma_{JK}^{-1}(\|(\widetilde B'\widetilde B/n) - I_K\| + \|\widehat S - S\|) = o_p(1)$ together imply that
\begin{equation}
 |s_J(\widehat S (\widetilde B'\widetilde B/n)^{-1/2}) - \sigma_{JK}| \leq \frac{1}{2}\sigma_{JK}
\end{equation}
wpa1. Let $\mathscr C_n$ be the intersection of $\mathscr A_n$ with the set on which this bound obtains. Then $\mathbb P(\mathscr C_n^c) = o(1)$. Clearly $(\widetilde B'\widetilde B/n)^{-1/2} \widehat S'$ has full column rank $J$ and $\widehat S (\widetilde B'\widetilde B/n)^{-1} \widehat S$ is invertible on $\mathscr C_n$.

\item On $\mathscr C_n \subseteq \mathscr A_n$ we have $\|(\widetilde B'\widetilde B/n)^{-1/2}\| \leq \sqrt 2$. Working on $\mathscr C_n$, similar arguments to those used to prove parts (ii) and (iii) yield
\begin{eqnarray}
 \|\widehat D - D\|
 & \leq & \|(\widetilde B'\widetilde B/n)^{-1/2} - I_K\| (  \|(\widetilde B'\widetilde B/n)^{-1/2}\| +1) + \|\widehat Q - Q\| \|(\widetilde B'\widetilde B/n)^{-1/2}\|  \\
 & \leq & (1 + \sqrt 2)\|(\widetilde B'\widetilde B/n)^{-1/2} - I_K\| + \sqrt 2\|\widehat Q - Q\|\,. \label{perturb bd 21}
\end{eqnarray}
Since $\widehat Q$ and $Q$ are orthogonal projection matrices, part (1.5) of Theorem 1.1 of \cite{LiLiCui} implies
\begin{equation} \label{lilicui bound}
 \|\widehat Q - Q\| \leq \sigma_{JK}^{-1} \|(\widetilde B'\widetilde B/n)^{-1/2} \widehat S' - S'\|
\end{equation}
on $\mathscr C_n$. Part (v) is then proved by substituting (\ref{lilicui bound}) and (\ref{schmitt bound}) into (\ref{perturb bd 21}).

\item Working on $\mathscr C_n$ (so we replace the generalized inverses with inverses), similar arguments used to prove part (v) yield
\begin{eqnarray}
 & & \| [\widehat S(\widetilde B'\widetilde B/n)^{-1} \widehat S']^{-1} \widehat S (\widetilde B'\widetilde B/n)^{-1} - [SS']^{-1}S \| \notag \\
 & \leq & \| [\widehat S(\widetilde B'\widetilde B/n)^{-1} \widehat S']^{-1} \widehat S (\widetilde B'\widetilde B/n)^{-1/2} \| \|(\widetilde B'\widetilde B/n)^{-1/2} - I_K\|  \\
 & & + \|[\widehat S(\widetilde B'\widetilde B/n)^{-1} \widehat S']^{-1} \widehat S (\widetilde B'\widetilde B/n)^{-1/2} - [SS']^{-1} S\| \notag \\
 & \leq & 2 \sigma_{JK}^{-1} \|(\widetilde B'\widetilde B/n)^{-1/2} - I_K\| + \|[\widehat S(\widetilde B'\widetilde B/n)^{-1} \widehat S']^{-1} \widehat S (\widetilde B'\widetilde B/n)^{-1/2} - [SS']^{-1} S\|\,. \label{perturb bd 41}
\end{eqnarray}
Theorem 3.1 of \cite{DingHuang} yields
\begin{equation} \label{dinghuang bound}
 \|[\widehat S(\widetilde B'\widetilde B/n)^{-1} \widehat S']^{-1} \widehat S (\widetilde B'\widetilde B/n)^{-1/2} - [SS']^{-1} S\| \lesssim \sigma_{JK}^{-2} \|(\widetilde B'\widetilde B/n)^{-1/2} \widehat S' - S'\|
\end{equation}
wpa1 by virtue of part (iii) and the condition $\sigma_{JK}^{-1}(\|(\widetilde B'\widetilde B/n) - I_K\| + \|\widehat S - S\|) = o_p(1)$. Substituting (\ref{dinghuang bound}) and (\ref{schmitt bound}) into (\ref{perturb bd 41}) establishes (vi).
\end{enumerate}
This completes the proof.
\end{proof}

\begin{lemma}\label{BH lemma}
Under Assumption \ref{b sieve}, if $\{X_i,Y_{2i}\}_{i=1}^n$ are i.i.d. then
\begin{enumerate}[(i)]
\item $ \|\widetilde B'(H_0 - \Psi c_J)/n\| \leq O_p(\sqrt{K/n}) \times \|h_0 - \pi_J h_0\|_\infty + \| \Pi_K T(h_0 - \pi_J h_0)\|_{L^2(X)}$
\item $\|\widetilde b^K(x) D \{ \widetilde B'(H_0 - \Psi c_J)/n - E[\widetilde b^K(X_i)(h_0(Y_{2i}) - \pi_J h_0(Y_{2i}))]\} \|_\infty = O_p(\sqrt{K (\log n)/n}) \times \|h_0 - \pi_J h_0\|_\infty$.
\end{enumerate}
\end{lemma}

\begin{proof}[Proof of Lemma \ref{BH lemma}]
We prove Lemma \ref{BH lemma} by part.
\begin{enumerate}[(i)]
\item First write
\begin{eqnarray}
 \|\widetilde B'(H_0 - \Psi c_J)/n\| & \leq & \|\widetilde B' (H_0 - \Psi c_J)/n - E[\widetilde b^K(X_i)(h_0(Y_{2i}) - \pi_J h_0(Y_{2i}))] \| \notag \\
 & & + \|E[\widetilde b^K(X_i)(h_0(Y_{2i} - \pi_J h_0(Y_{2i}))] \|
\end{eqnarray}
and note that
\begin{equation}
 \|E[\widetilde b^K(X_i)(h_0(Y_{2i}) - \pi_J h_0(Y_{2i}))] \|^2 = \| \Pi_K T(h_0 - \pi_J h_0)\|_{L^2(X)}^2\,.
\end{equation}
Finally,
\begin{equation}
 \|\widetilde B' (H_0 - \Psi c_J)/n - E[\widetilde b^K(X_i)(h_0(Y_{2i}) - \pi_J h_0(Y_{2i}))] \| = O_p( \sqrt{K/n}) \times \|h_0 - \pi_J h_0\|_\infty\,.
\end{equation}
by Markov's inequality under Assumption \ref{b sieve} and the fact that $\{X_i,Y_{2i}\}_{i=1}^n$ are i.i.d.

\item An argument similar to the proof of Theorem \ref{sup norm rate gen new} converts the problem of controlling the supremum that of controlling the maximum evaluated at finitely many points, where the collection of points has cardinality increasing polynomially in $n$. Let $\mathcal S_n'$ be the set of points. Also define
\begin{equation}
 \Delta_{i,J,K} = \widetilde b^K(X_i)(h_0(Y_{2i}) - \pi_J h_0(Y_{2i})) - E[\widetilde b^K(X_i)(h_0(Y_{2i}) - \pi_J h_0(Y_{2i}))]\}
\end{equation}
Then it suffices to show that sufficiently large $C$ may be chosen that
\begin{equation} \label{bern obj}
 (\# \mathcal S_n') \max_{x_n \in \mathcal S_n'} \mathbb P \left( \left| \sum_{i=1}^n n^{-1}\widetilde b^K(x_n) D\Delta_{i,J,K} \right| > C \|h_0 - \pi_J h_0\|_\infty \sqrt{K (\log n)/n}\right) = o(1)\,.
\end{equation}
The summands in (\ref{bern obj}) have mean zero (by the law of iterated expectations). Under Assumption \ref{b sieve} the summands in (\ref{bern obj}) are bounded uniformly for $x_n \in \mathcal S_n'$ by
\begin{equation} \label{bernstein 21}
 |n^{-1}\widetilde b^K(x_n) D\Delta_{i,J,K} | \lesssim  \frac{K}{n} \|h_0 - \pi_J h_0\|_\infty
\end{equation}
and have variance bounded uniformly for $x_n \in \mathcal S_n'$ by
\begin{eqnarray}
 E[(n^{-1}\widetilde b^K(x_n) D\Delta_{i,J,K})^2] & \leq & \|h_0 - \pi_J h_0\|_\infty^2 \times n^{-2} E[\widetilde b^K(x_n)'D\widetilde b^K(X_i)\widetilde b^K(X_i)' D \widetilde b^K(x_n)] \\
 & \lesssim & \|h_0 - \pi_J h_0\|_\infty^2 \times \frac{K}{n^2}\,. \label{bernstein 22}
\end{eqnarray}
The result follows for large enough $C$ by Bernstein's inequality for i.i.d. sequences using the bounds (\ref{bernstein 21}) and (\ref{bernstein 22}).

\end{enumerate}
This completes the proof.
\end{proof}

{\ 
\bibliographystyle{ecta}
\bibliography{srucrefs}

\begin{thebibliography}{55}
\newcommand{\enquote}[1]{``#1''}
\expandafter\ifx\csname natexlab\endcsname\relax\def\natexlab#1{#1}\fi

\bibitem[\protect\citeauthoryear{Ai and Chen}{Ai and Chen}{2003}]{AiChen2003}
\textsc{Ai, C. and X.~Chen} (2003): \enquote{Efficient Estimation of Models
  with Conditional Moment Restrictions Containing Unknown Functions,}
  \emph{Econometrica}, 71, 1795--1843.

\bibitem[\protect\citeauthoryear{Andrews}{Andrews}{2011}]{Andrews2011}
\textsc{Andrews, D. W.~K.} (2011): \enquote{Examples of L2-Complete and
  Boundedly-Complete Distributions,} \emph{Cowles Foundation Discussion Paper
  No. 1801}.

\bibitem[\protect\citeauthoryear{Belloni, Chernozhukov, and Kato}{Belloni
  et~al.}{2012}]{BCK2013}
\textsc{Belloni, A., V.~Chernozhukov, and K.~Kato} (2012): \enquote{On the
  Asymptotic Theory for Least Squares Series: Pointwise and Uniform Results,}
  Preprint, arXiv:1212.0442v1 [stat.ME].

\bibitem[\protect\citeauthoryear{Blundell, Chen, and Kristensen}{Blundell
  et~al.}{2007}]{Blundell2007}
\textsc{Blundell, R., X.~Chen, and D.~Kristensen} (2007):
  \enquote{Semi-Nonparametric IV Estimation of Shape-Invariant Engel Curves,}
  \emph{Econometrica}, 75, 1613--1669.

\bibitem[\protect\citeauthoryear{Blundell and Powell}{Blundell and
  Powell}{2003}]{BlundellPowell}
\textsc{Blundell, R. and J.~L. Powell} (2003): \enquote{Endogeneity in
  Nonparametric and Semiparametric Regression Models,} in \emph{Advances in
  Economics and Econometrics}, ed. by M.~Dewatripont, L.~P. Hansen, and S.~J.
  Turnovsky, Cambridge University Press, Cambridge.

\bibitem[\protect\citeauthoryear{Carroll and Hall}{Carroll and
  Hall}{1988}]{CarrollHall}
\textsc{Carroll, R.~J. and P.~Hall} (1988): \enquote{Optimal Rates of
  Convergence for Deconvolving a Density,} \emph{Journal of the American
  Statistical Association}, 83, 1184--1186.

\bibitem[\protect\citeauthoryear{Cattaneo and Farrell}{Cattaneo and
  Farrell}{2013}]{CattaneoFarrell}
\textsc{Cattaneo, M.~D. and M.~H. Farrell} (2013): \enquote{Optimal Convergence
  Rates, Bahadur Representation, and Asymptotic Normality of Partitioning
  Estimators,} \emph{Journal of Econometrics}, 174, 127--143.

\bibitem[\protect\citeauthoryear{Cavalier}{Cavalier}{2008}]{Cavalier2008}
\textsc{Cavalier, L.} (2008): \enquote{Nonparametric Statistical Inverse
  Problems,} \emph{Inverse Problems}, 24, 034004.

\bibitem[\protect\citeauthoryear{Cavalier, Golubev, Picard, and
  Tsybakov}{Cavalier et~al.}{2002}]{Cavalieretal2002}
\textsc{Cavalier, L., G.~K. Golubev, D.~Picard, and A.~B. Tsybakov} (2002):
  \enquote{Oracle Inequalities for Inverse Problems,} \emph{The Annals of
  Statistics}, 30, 843--874.

\bibitem[\protect\citeauthoryear{Cavalier and Hengartner}{Cavalier and
  Hengartner}{2005}]{CavalierHengartner}
\textsc{Cavalier, L. and N.~W. Hengartner} (2005): \enquote{Adaptive Estimation
  for Inverse Problems with Noisy Operators,} \emph{Inverse Problems}, 21,
  1345--1361.

\bibitem[\protect\citeauthoryear{Chen, Chernozhukov, Lee, and Newey}{Chen
  et~al.}{2013}]{Chen2012b}
\textsc{Chen, X., V.~Chernozhukov, S.~Lee, and W.~K. Newey} (2013):
  \enquote{Local Identification of Nonparametric and Semiparametric Models,}
  \emph{Econometrica, forthcoming}.

\bibitem[\protect\citeauthoryear{Chen and Ludvigson}{Chen and
  Ludvigson}{2009}]{Chen2009}
\textsc{Chen, X. and S.~C. Ludvigson} (2009): \enquote{Land of Addicts? An
  Empirical Investigation of Habit-Based Asset Pricing Models,} \emph{Journal
  of Applied Econometrics}, 24, 1057--1093.

\bibitem[\protect\citeauthoryear{Chen and Pouzo}{Chen and
  Pouzo}{2009}]{ChenPouzo2009}
\textsc{Chen, X. and D.~Pouzo} (2009): \enquote{Efficient Estimation of
  Semiparametric Conditional Moment Models with Possibly Nonsmooth Residuals,}
  \emph{Journal of Econometrics}, 152, 46--60.

\bibitem[\protect\citeauthoryear{Chen and Pouzo}{Chen and
  Pouzo}{2012}]{ChenPouzo2012}
---\hspace{-.1pt}---\hspace{-.1pt}--- (2012): \enquote{Estimation of
  Nonparametric Conditional Moment Models With Possibly Nonsmooth Generalized
  Residuals,} \emph{Econometrica}, 80, 277--321.

\bibitem[\protect\citeauthoryear{Chen and Reiss}{Chen and
  Reiss}{2011}]{ChenReiss}
\textsc{Chen, X. and M.~Reiss} (2011): \enquote{On Rate Optimality for
  Ill-Posed Inverse Problems in Econometrics,} \emph{Econometric Theory}, 27,
  497--521.

\bibitem[\protect\citeauthoryear{Cohen, Daubechies, Jawerth, and Vial}{Cohen
  et~al.}{1993{\natexlab{a}}}]{CDJV1993}
\textsc{Cohen, A., I.~Daubechies, B.~Jawerth, and P.~Vial}
  (1993{\natexlab{a}}): \enquote{Multiresolution Analysis, Wavelets and Fast
  Algorithms on an Interval,} \emph{Comptes Rendus de l'Acad\'{e}mie des
  Sciences Paris. S\'{e}rie 1, Math\'{e}matique}, 316, 417--421.

\bibitem[\protect\citeauthoryear{Cohen, Daubechies, and Vial}{Cohen
  et~al.}{1993{\natexlab{b}}}]{CDV1993}
\textsc{Cohen, A., I.~Daubechies, and P.~Vial} (1993{\natexlab{b}}):
  \enquote{Wavelets on the Interval and Fast Wavelet Transforms,} \emph{Applied
  and Computational Harmonic Analysis}, 1, 54--81.

\bibitem[\protect\citeauthoryear{Cohen, Hoffmann, and Reiss}{Cohen
  et~al.}{2004}]{Cohenetal2004}
\textsc{Cohen, A., M.~Hoffmann, and M.~Reiss} (2004): \enquote{Adaptive Wavelet
  Galerkin Methods for Linear Inverse Problems,} \emph{SIAM Journal on
  Numerical Analysis}, 42, 1479--1501.

\bibitem[\protect\citeauthoryear{Darolles, Fan, Florens, and Renault}{Darolles
  et~al.}{2011}]{Darollesetal2011}
\textsc{Darolles, S., Y.~Fan, J.-P. Florens, and E.~Renault} (2011):
  \enquote{Nonparametric Instrumental Regression,} \emph{Econometrica}, 79,
  1541--1565.

\bibitem[\protect\citeauthoryear{{De Boor}}{{De Boor}}{2001}]{deBoor2001}
\textsc{{De Boor}, C.} (2001): \emph{A Practical Guide to Splines},
  Springer-Verlag, New York.

\bibitem[\protect\citeauthoryear{de~Jong}{de~Jong}{2002}]{deJong2002}
\textsc{de~Jong, R.~M.} (2002): \enquote{A Note on ``Convergence Rates and
  Asymptotic Normality for Series Estimators�'': Uniform Convergence Rates,}
  \emph{Journal of Econometrics}, 111, 1--9.

\bibitem[\protect\citeauthoryear{D'Haultfoeuille}{D'Haultfoeuille}{2011}]{D'Haultfoeuille}
\textsc{D'Haultfoeuille, X.} (2011): \enquote{On the Completeness Condition in
  Nonparametric instrumental regression,} \emph{Econometric Theory}, 27,
  460--471.

\bibitem[\protect\citeauthoryear{Ding and Huang}{Ding and
  Huang}{1997}]{DingHuang}
\textsc{Ding, J. and L.~Huang} (1997): \enquote{On the Continuity of
  Generalized Inverses of Linear Operators in {H}ilbert Spaces,} \emph{Linear
  Algebra and its Applications}, 262, 229--242.

\bibitem[\protect\citeauthoryear{Doukhan, Massart, and Rio}{Doukhan
  et~al.}{1995}]{DoukhanMassartRio}
\textsc{Doukhan, P., P.~Massart, and E.~Rio} (1995): \enquote{Invariance
  Principles for Absolutely Regular Empirical Processes,} \emph{Annales de
  l'Institut Henri Poincar\'{e} (B) Probabilit\'{e}s et Statistiques}, 31,
  393--427.

\bibitem[\protect\citeauthoryear{Efromovich and Koltchinskii}{Efromovich and
  Koltchinskii}{2001}]{EfromovichKoltchinskii}
\textsc{Efromovich, S. and V.~Koltchinskii} (2001): \enquote{On Inverse
  Problems with Unknown Operators,} \emph{IEEE Transactions on Information
  Theory}, 47, 2876--2894.

\bibitem[\protect\citeauthoryear{Fan}{Fan}{1991}]{Fan1991}
\textsc{Fan, J.} (1991): \enquote{On the Optimal Rates of Convergence for
  Nonparametric Deconvolution Problems,} \emph{The Annals of Statistics}, 19,
  1257--1272.

\bibitem[\protect\citeauthoryear{Florens and Simoni}{Florens and
  Simoni}{2012}]{FlorensSimoni}
\textsc{Florens, J.-P. and A.~Simoni} (2012): \enquote{Nonparametric estimation
  of an instrumental variables regression: a quasi-Bayesian approach based on
  regularized posterior,} \emph{Journal of Econometrics}, 170, 458--475.

\bibitem[\protect\citeauthoryear{Gagliardini and Scaillet}{Gagliardini and
  Scaillet}{2012}]{GagliardiniScaillet}
\textsc{Gagliardini, P. and O.~Scaillet} (2012): \enquote{Tikhonov
  Regularization for Nonparametric Instrumental Variable Estimators,}
  \emph{Journal of Econometrics}, 167, 61--75.

\bibitem[\protect\citeauthoryear{Hall and Horowitz}{Hall and
  Horowitz}{2005}]{HallHorowitz}
\textsc{Hall, P. and J.~L. Horowitz} (2005): \enquote{Nonparametric Methods for
  Inference in the Presence of Instrumental Variables,} \emph{The Annals of
  Statistics}, 33, 2904--2929.

\bibitem[\protect\citeauthoryear{Hall and Meister}{Hall and
  Meister}{2007}]{HallMeister}
\textsc{Hall, P. and A.~Meister} (2007): \enquote{A Ridge-Parameter Approach to
  Deconvolution,} \emph{The Annals of Statistics}, 35, 1535--1558.

\bibitem[\protect\citeauthoryear{Hansen}{Hansen}{2008}]{hansen-ET}
\textsc{Hansen, B.} (2008): \enquote{Uniform convergence rates for kernel
  estimation with dependent data,} \emph{Econometric Theory}, 24, 726--748.

\bibitem[\protect\citeauthoryear{Hoffmann and Reiss}{Hoffmann and
  Reiss}{2008}]{HoffmannReiss}
\textsc{Hoffmann, M. and M.~Reiss} (2008): \enquote{Nonlinear Estimation for
  Linear Inverse Problems with Error in the Operator,} \emph{The Annals of
  Statistics}, 36, 310--336.

\bibitem[\protect\citeauthoryear{Horowitz}{Horowitz}{2011}]{Horowitz2011}
\textsc{Horowitz, J.~L.} (2011): \enquote{Applied Nonparametric Instrumental
  Variables Estimation,} \emph{Econometrica}, 79, 347--394.

\bibitem[\protect\citeauthoryear{Huang}{Huang}{1998}]{Huang1998}
\textsc{Huang, J.~Z.} (1998): \enquote{Projection Estimation in Multiple
  Regression with Application to Functional {ANOVA} Models,} \emph{The Annals
  of Statistics}, 26, 242--272.

\bibitem[\protect\citeauthoryear{Huang}{Huang}{2003}]{Huang2003}
---\hspace{-.1pt}---\hspace{-.1pt}--- (2003): \enquote{Local Asymptotics for
  Polynomial Spline Regression,} \emph{The Annals of Statistics}, 31,
  1600--1635.

\bibitem[\protect\citeauthoryear{Johnstone}{Johnstone}{2013}]{Johnstone2013}
\textsc{Johnstone, I.~M.} (2013): \enquote{Gaussian Estimation: Sequence and
  Wavelet Models,} Manuscript.

\bibitem[\protect\citeauthoryear{Kress}{Kress}{1999}]{Kress}
\textsc{Kress, R.} (1999): \emph{Linear Integral Equations}, Springer-Verlag,
  Berlin.

\bibitem[\protect\citeauthoryear{Lee and Robinson}{Lee and
  Robinson}{2013}]{LeeRobinson}
\textsc{Lee, J. and P.~Robinson} (2013): \enquote{Series Estimation Under
  Cross-sectional Dependence,} Preprint, London School of Economics.

\bibitem[\protect\citeauthoryear{Li, Li, and Cui}{Li et~al.}{2013}]{LiLiCui}
\textsc{Li, B., W.~Li, and L.~Cui} (2013): \enquote{New Bounds for Perturbation
  of the Orthogonal Projection,} \emph{Calcolo}, 50, 69--78.

\bibitem[\protect\citeauthoryear{Liao and Jiang}{Liao and
  Jiang}{2011}]{LiaoJiang}
\textsc{Liao, Y. and W.~Jiang} (2011): \enquote{Posterior Consistency of
  Nonparametric Conditional Moment Restricted Models,} \emph{The Annals of
  Statistics}, 39, 3003--3031.

\bibitem[\protect\citeauthoryear{Loubes and Marteau}{Loubes and
  Marteau}{2012}]{LoubesMarteau}
\textsc{Loubes, J.-M. and C.~Marteau} (2012): \enquote{Adaptive Estimation for
  an Inverse Regression model with Unknown Operator,} \emph{Statistics \& Risk
  Modeling}, 29, 215--242.

\bibitem[\protect\citeauthoryear{Lounici and Nickl}{Lounici and
  Nickl}{2011}]{LouniciNickl}
\textsc{Lounici, K. and R.~Nickl} (2011): \enquote{Global Uniform Risk Bounds
  for Wavelet Deconvolution Estimators,} \emph{The Annals of Statistics}, 39,
  201--231.

\bibitem[\protect\citeauthoryear{Masry}{Masry}{1996}]{masry}
\textsc{Masry, E.} (1996): \enquote{Multivariate local polynomial regression
  for time series: uniform strong consistency and rates,} \emph{Journal of Time
  Series Analysis}, 17, 571--599.

\bibitem[\protect\citeauthoryear{Newey}{Newey}{1997}]{Newey1997}
\textsc{Newey, W.~K.} (1997): \enquote{Convergence Rates and Asymptotic
  Normality for Series Estimators,} \emph{Journal of Econometrics}, 79,
  147--168.

\bibitem[\protect\citeauthoryear{Newey and Powell}{Newey and
  Powell}{2003}]{NeweyPowell}
\textsc{Newey, W.~K. and J.~L. Powell} (2003): \enquote{Instrumental Variable
  Estimation of Nonparametric Models,} \emph{Econometrica}, 71, 1565--1578.

\bibitem[\protect\citeauthoryear{Schmitt}{Schmitt}{1992}]{Schmitt}
\textsc{Schmitt, B.~A.} (1992): \enquote{Perturbation Bounds for Matrix Square
  Roots and Pythagorean Sums,} \emph{Linear Algebra and its Applications}, 174,
  215--227.

\bibitem[\protect\citeauthoryear{Schumacker}{Schumacker}{2007}]{Schumacker2007}
\textsc{Schumacker, L.~L.} (2007): \emph{Spline Functions: Basic Theory},
  Cambridge University Press, Cambridge.

\bibitem[\protect\citeauthoryear{Song}{Song}{2008}]{Song2008}
\textsc{Song, K.} (2008): \enquote{Uniform Convergence of Series Estimators
  over Function Spaces,} \emph{Econometric Theory}, 24, 1463--1499.

\bibitem[\protect\citeauthoryear{Stone}{Stone}{1982}]{Stone1982}
\textsc{Stone, C.~J.} (1982): \enquote{Optimal Global Rates of Convergence for
  Nonparametric Regression,} \emph{The Annals of Statistics}, 10, 1040--1053.

\bibitem[\protect\citeauthoryear{Triebel}{Triebel}{2006}]{Triebel2006}
\textsc{Triebel, H.} (2006): \emph{Theory of Function Spaces III},
  Birkh\"auser, Basel.

\bibitem[\protect\citeauthoryear{Triebel}{Triebel}{2008}]{Triebel2008}
---\hspace{-.1pt}---\hspace{-.1pt}--- (2008): \emph{Function Spaces and
  Wavelets on Domains}, European Mathematical Society, Z\"{u}rich.

\bibitem[\protect\citeauthoryear{Tropp}{Tropp}{2012}]{Tropp2012}
\textsc{Tropp, J.~A.} (2012): \enquote{User-Friendly Tail Bounds for Sums of
  Random Matrices,} \emph{Foundations of Computational Mathematics}, 12,
  389--434.

\bibitem[\protect\citeauthoryear{Tsybakov}{Tsybakov}{2009}]{Tsybakov2009}
\textsc{Tsybakov, A.~B.} (2009): \emph{Introduction to Nonparametric
  Estimation}, Springer, New York.

\bibitem[\protect\citeauthoryear{van~de Geer}{van~de
  Geer}{1995}]{vandeGeer1995}
\textsc{van~de Geer, S.} (1995): \enquote{Exponential Inequalities for
  Martingales, with Application to Maximum Likelihood Estimation for Counting
  Processes,} \emph{The Annals of Statistics}, 23, 1779--1801.

\bibitem[\protect\citeauthoryear{Zhang}{Zhang}{1990}]{Zhang}
\textsc{Zhang, C.-H.} (1990): \enquote{Fourier methods for estimating mixing
  densities and distributions,} \emph{The Annals of Statistics}, 18, 806--831.

\end{thebibliography}
}

\end{document}